%% file: lseb_v2.tex
\documentclass[12pt, a4paper, final]{article} 			
\pdfoutput=1
\usepackage[utf8]{inputenc}  							
\usepackage{amsmath}
\usepackage{amssymb}
\usepackage{amsthm}
\usepackage{dsfont}
\usepackage{floatflt} 
\usepackage{color}
\usepackage{enumerate}									
\usepackage{xspace}										
\usepackage{caption}

\usepackage{ragged2e}									
\usepackage[style = numeric,sorting = nyt,sortcites,firstinits = true,isbn = false,doi = false,clearlang = true,backend = bibtex]{biblatex}
\usepackage{csquotes}	
\DeclareRedundantLanguages{English,english,german,french}{english,german,ngerman,french} 

\renewbibmacro{in:}{%
	\ifentrytype{article}{}{\printtext{\bibstring{in}\intitlepunct}}}		
	
\usepackage{graphicx}				
\usepackage{enumitem}

\usepackage{bbm}										
\usepackage{calligra}
\DeclareMathAlphabet{\mathcalligra}{T1}{calligra}{m}{n}	
\DeclareFontShape{T1}{calligra}{m}{n}{<->s*[2]callig15}{}

\usepackage{geometry}
\allowdisplaybreaks[1]		

\usepackage{aliascnt}
\usepackage[pdfpagelabels=true,plainpages=false, hidelinks]{hyperref}

\begin{document}

\newtheorem{theorem}{Theorem}[section]
\newtheorem{theorema}{Theorem}						
\renewcommand{\thetheorema}{\Alph{theorema}}
\providecommand*{\theoremaautorefname}{Theorem} 
\newtheorem*{theorem*}{Theorem}

\newtheorem*{maintheorem}{Main Theorem}

\newtheorem*{conj}{Conjecture}

\newaliascnt{lemma}{theorem}
\newtheorem{lemma}[lemma]{Lemma}
\aliascntresetthe{lemma}
\providecommand*{\lemmaautorefname}{Lemma} 

\newaliascnt{cor}{theorem}
\newtheorem{cor}[cor]{Corollary}
\aliascntresetthe{cor}
\providecommand*{\corautorefname}{Corollary} 

\newtheorem*{rem}{Remark}

\newaliascnt{prop}{theorem}
\newtheorem{prop}[prop]{Proposition}
\aliascntresetthe{prop}
\providecommand*{\propautorefname}{Proposition} 

\theoremstyle{definition}


\newcommand{\enmath}[1]{\ensuremath{#1}\xspace} 
\newcommand{\defeq}{\stackrel{\mathrm{def}}{=}}
\newcommand{\RR}{\mathbb{R}}
\newcommand{\ZZ}{\mathbb{Z}}
\newcommand{\NN}{\mathbb{N}}
\newcommand{\CC}{\mathbb{C}}
\newcommand{\HH}{\mathbb{H}}
\newcommand{\Q}{\mathcal{Q}}
\newcommand{\RP}{\mathbb{RP}^2}
\newcommand{\SO}[1]{\enmath{\operatorname{SO}(#1)}}
\newcommand{\OO}[1]{\enmath{\operatorname{O}(#1)}}
\newcommand{\diagfour}[4]{\begin{pmatrix} #1 & 0 & 0 & 0 \\ 0 & #2 & 0 & 0 \\ 0 & 0 & #3 & 0 \\ 0 & 0 & 0 & #4 \end{pmatrix}}
\newcommand{\diag}[1]{\operatorname{diag}\left(#1\right)}
\newcommand{\fix}[1]{\operatorname{Fix}\left(#1\right)}
\newcommand{\tr}[1]{\operatorname{tr}\left(#1\right)}
\newcommand{\FF}{\mathcal F}
\newcommand{\GG}{\mathcal G}
\newcommand{\Hh}{\mathcal H}
\newcommand{\CG}{C_G^\infty\left(\RR^8 \times \RR, \RR^8 \right)}
\newcommand{\VV}{\mathcal{V}_0}
\newcommand{\pvf}[1]{\mathcal{P}_{#1}}
\newcommand{\ROT}[1]{\textbf{\textcolor{red}{(#1)}}}			
\newcommand{\ii}{\mathbf{i}}
\newcommand{\jj}{\mathbf{j}}
\newcommand{\kk}{\mathbf{k}}
\newcommand{\spin}{\enmath{\operatorname{Spin}_4}}
\newcommand{\Gab}{\enmath{G_{a,b}}}
\newcommand{\Ga}{\enmath{\mathbf{G}_a}}
\newcommand{\Hab}{\enmath{H_{a,b}}}
\newcommand{\Ha}{\enmath{\mathbf{H}_a}}
\newcommand{\A}{\mathbb{A}}
\newcommand{\rkr}[1]{\ZZ / #1 \ZZ}
\newcommand{\m}{\mathcalligra{m}}
\newcommand{\bl}{\bullet}
\newcommand{\rr}{\mathbf{r}}

\renewcommand\autocite[1]{\textcite{#1}}

\numberwithin{equation}{section} 
\renewcommand{\theequation}{\arabic{section}.\arabic{equation}} 


\title{Equivariant bifurcations in $4$-dimensional \\ fixed point spaces}
\author{Reiner Lauterbach and Sören Schwenker}
\date{}
\maketitle
\thanks{In honor of Marty Golubitsky on the occasion of his seventieth birthday.}

\input{lseb_in_v2}		

\input{lseb_mr_v2}

\input{lseb_pn_v2}

\input{lseb_gr_v2}

\input{lseb_es_v2}

\input{lseb_bi_v2}

\input{lseb_sc}

\input{lseb_fg_v2}

\input{lseb_ap_v2}		

\newpage
\begingroup
\RaggedRight
\printbibliography
\endgroup
\phantomsection
\addcontentsline{toc}{section}{References} 
\end{document}

%% file: lseb_in_v2.tex

\begin{abstract}
	In this paper we continue the study of group representations which are counterexamples to the Ize conjecture. As in the previous papers by \autocite{lauterbach2014equivariant} and \autocite{lauterbach2010do} we find new infinite series of finite groups leading to such counterexamples. These new series are quite different from the previous ones, for example the group orders do not form an arithmetic progression. However, as before we find Lie groups which contain all these groups. This additional structure was observed, but not used in the previous studies of this problem. 
	Here we also investigate the related bifurcations. To a large extent, these are closely related to the presence of mentioned  compact Lie group containing the finite groups. This might give a tool to study the bifurcations related to all low dimensional counterexamples of the Ize conjecture. It also gives an indication of where we can expect to find examples where the bifurcation behavior is different from what we have seen in the known examples.  
\end{abstract}

\section*{Introduction}
\autocite{lauterbach2010do} have looked at the Ize conjecture:
\begin{conj}[J. Ize]
	Let $V$ be a real, linear and absolutely irreducible representation of a finite group or a compact Lie group $G$. Then there exists an isotropy subgroup $H \le G$ with odd-dimensional fixed point space.
\end{conj}

\noindent
They proved that this conjecture is not true by presenting three infinite families of finite groups acting on $\RR^4$, such that any of these groups has only nontrivial isotropy subgroups whose corresponding fixed point spaces are two-dimensional. They also show that for equivariant bifurcations with any group in the first two families at least one of the nontrivial isotropy types is generically symmetry breaking (in the sense of \autocite{field1990symmetry}). In their construction each of these families relates to a compact Lie group, which contains all the groups in the family, however these Lie groups do not a play a substantial role in the analysis. \\
Concerning dimensions of representation spaces which are small multiples of 4 they provide tables presenting computational results on counterexamples to the Ize conjecture including the three mentioned families. It turns out that there are, besides the three families, many more potential counterexamples to the Ize conjecture (however there are no proofs yet). The bifurcation question for all of these groups is completely open. In \autocite{lauterbach2014equivariant} the  third family is analysed including the question concerning the generic bifurcations. Based on this information a new family of infinitely many finite groups acting on $\RR^8$ is constructed. For both cases in dimension 4 and in dimension 8 it is shown that generically the (only) nontrivial isotropy type is symmetry breaking in the sense of Field and Richardson.   Again there is a compact Lie group which plays no visible role in this context. In part 4 of Theorem B in \cite{lauterbach2014equivariant}, \citeauthor{lauterbach2014equivariant} stated that this Lie group is a counter example to the Ize conjecture. However no proof for this statement is provided and in fact it is not correct as one can easily see. \\
In this paper we investigate infinite families of finite groups whose orders do not form an arithmetic progression as in the previous examples. Moreover we construct a new family acting on $\RR^4$ and based on this family a second family acting on $\RR^8$ which has a single nontrivial isotropy type and the dimension of its fixed point space is four dimensional. We prove that this isotropy type is generically symmetry breaking. The proofs are substantially different from the previous ones, here we make essential use of the Lie groups containing the groups in the family. The general question whether counterexamples to Ize's conjecture possess isotropy types which are generically symmetry breaking is open, but our technique might provide a tool to either construct counterexamples or to provide proofs.     

%% file: lseb_mr_v2.tex
\section{Main results}
\label{lseb_mr}
\autocite{lauterbach2010do} have constructed three families of groups of orders $16\ell$ with $\ell\in 2\NN+1$, acting absolutely irreducibly on $\RR^4$ and leading to counterexamples to the Ize conjecture. In \cite{lauterbach2014equivariant}, \citeauthor{lauterbach2014equivariant} continues this work and constructs a family of groups of order $64\ell$ with $\ell \in 2\NN+1$ acting absolutely irreducibly on $\RR^8$ with only even-dimensional fixed point subspaces. In this paper we construct groups of order $8\m$ where $\m$ is odd and of the form 
\begin{equation}
\label{decm}
\m=a\cdot b \quad \text{with } a, b \in 2\NN +1 \text{ and } \gcd (a,b)  = 1
\end{equation}
(This sequence is listed in the On-Line Encyclopedia of Integer Sequences as sequence A061346 (\url{http://www.oeis.org})). These groups act absolutely irreducibly on $\RR^4$ and we use them to construct groups twice their size acting absolutely irreducibly on $\RR^8$. For this step $a$ needs to be of a special form guaranteeing the existence of square roots of $-1$ modulo $a$:
\begin{prop}
	\label{prime}
	Let $a_i = 1 \mod 4$ be prime and $s_i \in \NN$ for $i=1, \ldots , r$. Furthermore let
	\[a = \prod_{i=1}^r a_i^{s_i}. \]
	Then there exists $\rho \in \NN$ such that
	\[ \rho^2 = -1 \mod a.\]
\end{prop}

\noindent
For further use we denote the set of such $a$ by $\A$:
\begin{align*}
	\label{eqA}
	\A	&= \left\lbrace \prod_{i=1}^r a_i^{s_i} \mid r \in \NN; a_i \text{ prime}, a_i = 1 \mod 4, s_i \in \NN \text{ for } i=1, \ldots , r \right\rbrace \\
	&= \lbrace 5, 13, 17, 25, 29, 37, 41, 53, 61, 65, \ldots \rbrace \notag
\end{align*}
(The sequence of these $\m$ is a subsequence of the one listed as sequence A257591 in the On-Line Encyclopedia of Integer Sequences (\url{http://www.oeis.org})). The groups acting on $\RR^8$ have precisely one nontrivial isotropy type which has a $4$-dimensional fixed point space and therefore lead to counterexamples to the Ize conjecture in dimension $8$.

The construction in \autocite{lauterbach2010do} relies heavily on the \emph{biquaternionic characterization} of elements in $\SO{4}$ presented in \autocite{conway2003on}. We briefly recall the necessary notations. Denote the space of quaternions with the standard basis $\lbrace1,\ii,\jj,\kk \rbrace$ where $\ii^2 = \jj^2 = \kk^2 = -1$ by $\HH$ and let $\Q \subset \HH$ be the set of unitary quaternions. The group of ordered pairs of such quaternions $\Q\times \Q$ is isomorphic to the $4$-dimensional spin group \spin. Identifying $\HH$ with $\RR^4$ in the obvious manner,
\[ x= \left(  x_1, x_2, x_3, x_4 \right)^T \leftrightarrow x_1 + \ii x_2 + \jj x_3 + \kk x_4 , \]
these pairs correspond to elements in \SO{4} via
\[ \widetilde{[l,r]}\colon x \mapsto \bar{l}xr.\]
\autocite{conway2003on} show that this is a two-to-one map on \SO{4} where the needed identification is $\widetilde{[1,1]}=\widetilde{[-1,-1]}$ yielding $\widetilde{[l,r]}=\widetilde{[-l,-r]}$, which obviously both map a point $x$ to the same image point.  The authors have used this characterization to classify the closed subgroups of \SO{4} in terms of the biquaternionic notation.

It is a subtle yet very important observation that this map is -- when taking the identification into account -- a bijection but it is not a group homomorphism. Following \autocite{chillingworth2015molien} we define a similar map via 
\begin{equation}
\label{smallrep}
[l,r]\colon x \mapsto lx \bar{r}.
\end{equation}
This provides an isomorphism $\Q \times \Q \to \SO{4}$ using the same identification $[1,1] = [-1,-1]$. Therefore we may view it as a group representation. Note that this isomorphism corresponds to the application of $\widetilde{[l,r]}^{-1}$ with the map given in \autocite{conway2003on} as $\bar{l} = l^{-1}$ for unitary quaternions. The application of $\widetilde{[l,r]}$ therefore yields an anti-representation. The tilde notation is obsolete from now on.

In a similar manner we can construct a map $\Q \times \Q \to \OO{4} \setminus \SO{4}$ using
\[ *[l,r] \colon x \mapsto l \bar{x} \bar{r}. \]
We do not need the explicit definition of this map but the fact that it is two-to-one as well turns out to be helpful.

We may now define the groups, we want to study in more detail. Let
\[ e_p = e^{\frac{\pi \ii}{p}} \]
be one of the primitive $p$-th root of $-1$ in $\CC$ and denote a group that is generated by the elements $g_1, g_2, \ldots$ by $\langle g_1, g_2, \ldots \rangle.$ Choose $a, b\in 2\NN+1$ such that they are relatively prime and define
\begin{equation}
	\label{H_ab}
	\Hab = \left\langle [e_a,1], [1,e_b], [1,\jj], [\jj,1] \right\rangle.
\end{equation}
We summarize results on the structure and the $4$-dimensional representation of these groups in the following theorems.

\begin{theorem}
	\label{thmsmallgroups}
	\begin{enumerate}[label={(\arabic*)},ref={\thecor~(\arabic*)}]
		\item For each odd $\m \in \NN$ and each decomposition $\m=a \cdot b$ as in \eqref{decm} \Hab forms a subgroup of \SO{4} of order $8\m$.
		\item Let $b<b'$ where $b'$ is odd and relatively prime to $a$. If $b$ divides $b'$, \Hab is a subgroup of $H_{a,b'}$, i.e.
		\[\Hab \le H_{a,b'}. \]
		\item The action of \Hab on $\HH$ as defined in \eqref{smallrep} is absolutely irreducible. It has precisely two nontrivial isotropy types. The corresponding fixed point spaces are $2$-dimensional.
	\end{enumerate}
\end{theorem}

\noindent
We use this construction to define the family $\Hh_a$ for each $a \in 2\NN+1$:
\begin{equation*}
\label{smallfamily}
\Hh_a = \left\lbrace \Hab \mid b \in 2\NN+1 \text{ and } \gcd(a,b)=1 \right\rbrace.
\end{equation*}
The last result of \autoref{thmsmallgroups} allows us to generate a one-dimensional Lie group for each suitable $a$ as follows:

\begin{theorem}
	\label{thmsmallgroups2}
	\begin{enumerate}[label={(\arabic*)},ref={\thecor~(\arabic*)}]
		\item Let $a \in 2\NN+1$. Then the set
		\begin{equation*}
		\Ha = \overline{\bigcup_{H \in \Hh_a} H} = \left\langle [e_a,1], [1,\jj], [\jj,1], [1,e^{\psi\ii}] \mid \psi \in S^1 \right\rangle
		\end{equation*}
		forms a compact Lie group of dimension $1$. Its action on $\HH$ is absolutely irreducible and it possesses isotropy subgroups with one-dimensional fixed point space.
		\item Let $a, a' \in 2\NN+1$ odd with $a < a'$. If $a$ divides $a'$, \Ha is a subgroup of $\mathbf{H}_{a'}$, i.e.
		\[ \Ha \le \mathbf{H}_{a'}. \]
	\end{enumerate}
\end{theorem}

\noindent
In the same manner this gives rise to a Lie group of dimension $2$:
\begin{theorem}
	\label{thmsmallgroups3}
	The set
	\begin{equation*}
	\label{smalllie2}
	\mathbf{H} = \overline{\bigcup_{a \in 2\NN+1} \Ha} = \left\langle [1,\jj], [\jj,1], [e^{\ii \phi},1], [1,e^{\ii \psi}] \mid \phi, \psi \in S^1 \right\rangle
	\end{equation*}
	forms a compact Lie group of dimension $2$.
\end{theorem}

To perform the final step in the construction of the groups acting on $\RR^8$ we need the matrix representatives of the generating elements of the groups \Hab and denote them as follows:
\[ \left[e_a,1\right] \leftrightarrow c, \quad \left[1,e_b\right] \leftrightarrow d, \quad [1,\jj] \leftrightarrow q, \quad [\jj,1] \leftrightarrow s. \]
We then look at $8$-dimensional representations of the groups constructed so far and extend them so that the representation becomes absolutely irreducible. Let $a \in \A$ and $b \in 2\NN+1$ such that $a$ and $b$ are relatively prime as before. Choose $\rho$ as in Proposition~\ref{prime}. Without loss of generality we may assume $\rho$ to be odd. If $\rho^2 = -1 \mod a$ then the same holds for $-\rho$. Since $a$ is odd, either $\rho$ or $-\rho$ is odd. We define a group as follows:
\begin{equation}
\label{smallmatrep} 
\tilde{H} = \left\langle [e_a,1]^\rho, [1,e_b], [1,\jj]^3, [\jj,1] \right\rangle. 
\end{equation}
This is obviously a subgroup of \Hab. Furthermore
\begin{equation*}
\left(\left([e_a,1]^\rho \right)^{-1}\right)^\rho = [e_a,1] \quad \text{and} \quad
[\jj,1]^2 [1,\jj]^3 = [1,\jj]
\end{equation*}
so the other  inclusion holds as well and therefore the groups are equal: $\tilde{H} = \Hab.$ Hence we can define a representation of \Hab on $\RR^4$ where $[e_a,1]$ and $[1,\jj]$ act as $[e_a,1]^\rho$ and $[1,\jj]^3$ respectively. 

We are interested in the direct sum of these two representations. It defines a group action of \Hab on $\RR^8$ which is obviously reducible. To guarantee absolute irreducibility we need to supplement the set of generators of \Hab with an element $v$ which exchanges the blocks of the two representations. We define its action on $\RR^8$ as follows: let $x, y \in \RR^4$ then
\[ v \begin{pmatrix}
x \\
y
\end{pmatrix} = \begin{pmatrix}
\mathbbm{1} _4 y \\
[ \jj , 1 ] x
\end{pmatrix}. \]
Since $v$ cannot be displayed properly in terms of pairs of unitary quaternions we focus on the matrix representation from now on. To meet the assumptions we have made on the $8$-dimensional representation we define the matrix generators as block matrices as follows (see \eqref{smallmatrep} and \eqref{MatrixRep}):
\begin{equation*}
\label{bigMatrixRep}
\begin{alignedat}{5}
C(a) = C 	&= \begin{pmatrix}
c & 0 \\
0 & c^\rho
\end{pmatrix}& &\quad
D(b)& & =D& 	&= \begin{pmatrix}
d & 0 \\
0 & d
\end{pmatrix} \\
Q			&= \begin{pmatrix}
q & 0 \\
0 & -q
\end{pmatrix}& &\quad
& &\mathrel{\phantom{=}} S &			&= \begin{pmatrix}
s & 0 \\
0 & s
\end{pmatrix} \\
V 			& = \begin{pmatrix}
0 & \mathbbm{1} _4 \\
s & 0
\end{pmatrix}.
\end{alignedat}
\end{equation*}
With these we define the groups acting on $\RR^8$ in terms of matrix generators
\[ \Gab = \left\langle C, D, Q, S, V \right\rangle \]
where the dependence on $a$ and $b$ lies in the matrices $C$ and $D$.

We obtain similar results on the structure and the $8$-dimensional representations as in the $4$-dimensional case.
\begin{theorem}
	\label{thmgroups}
	\begin{enumerate}[label={(\arabic*)},ref={\thecor~(\arabic*)}]
		\item For each $a \in \A$ and $b \in 2\NN+1$ with $\gcd(a,b)=1$ as above the group \Gab forms a subgroup of \OO{8} of order $16\m$ where $\m = a \cdot b$.
		\item Let $b<b'$ where $b'$ is odd and relatively prime to $a$. If $b$ divides $b'$, \Gab is a subgroup of $G_{a,b'}$, i.e.
		\[ \Gab \le G_{a,b'}. \]
		\item The natural $8$-dimensional representation of \Gab is absolutely irreducible. It has precisely one nontrivial isotropy type. The corresponding fixed point space is $4$-dimensional.
	\end{enumerate}
\end{theorem}

\noindent
We define families of these groups for every suitable $a$ as well
\begin{equation*}
\label{family}
\GG_a = \left\lbrace \Gab \mid b \ne 1 \text{ odd and } \gcd(a,b)=1 \right\rbrace.
\end{equation*}
In a similar manner as in the $4$-dimensional case we can generate compact Lie groups of dimension $1$ from the groups \Gab for every $a \in \A$. To do so we adapt notation of the generating matrices to characterize arbitrary rotations. Denote the $2$-dimensional rotation matrix by an angle $\psi$ by $\rr(\psi)$ and write
\[
d(\psi)	=
\begin{pmatrix}
\rr(- \psi)	& 0 \\
0			& \rr(\psi)
\end{pmatrix},\\
\]
where $\psi \in S^1$ (compare this to \eqref{MatrixRep}) and
\[
D(\psi)	=
\begin{pmatrix}
d(\psi)	& 0 \\
0		& d(\psi)
\end{pmatrix}.
\]

\begin{theorem}
	\label{thmgroups2}
	Let $a \in \A$. Then the set
	\begin{equation*}
	\label{liegroup}
	\Ga = \overline{\bigcup_{G\in \GG_a} H } = \left\langle C, Q, S, V, D(\psi) \mid \psi \in S^1 \right\rangle
	\end{equation*}
	forms a compact Lie group of dimension 1. Its natural $8$-dimensional representation is absolutely irreducible and it possesses isotropy subgroups with one-dimensional fixed point spaces.
\end{theorem}

Furthermore we investigate the equivariant structure and the bifurcation behaviour of \Gab-symmetric systems and obtain the final result.
\begin{theorem}
	\label{thmes}
	Let $a \in \A$ with $a>5$ and $G \in \GG_a$. The $8$-dimensional representation of $G$ has no quadratic equivariants. The space of cubic equivariants $P_{G}^3\left(\RR^8, \RR^8 \right)$ is $5$-dimensional. A basis is given by the maps $E_1, \ldots , E_5$ (see \autoref{appce}). Furthermore these are equivariant with respect to the Lie groups \Ga.
\end{theorem}

\begin{maintheorem}
	For the natural $8$-dimensional representation of $G \in \Ga$ with $5<a$ the only nontrivial isotropy type is generically symmetry breaking. Systems that are symmetric with respect to this representation generically have nontrivial symmetry breaking branches of steady states that are hyperbolic within the fixed point spaces.
\end{maintheorem}

\begin{rem}
	The bifurcation result holds true for the groups \Hab as well. The two nontrivial isotropy types of the $4$-dimensional representation are generically symmetry breaking. The proof uses exactly the same techniques as the proof for the main theorem. But as this is no new result on counterexamples to the Ize conjecture in $4$-dimensional representations, we omit the details and only present the proof for the $8$-dimensional case.
\end{rem}

%% file: lseb_pn_v2.tex
\section[Prime numbers of the form \texorpdfstring{$1 \mod 4$}{1 mod 4}]{Prime numbers of the form $\boldsymbol{1 \mod 4}$}
To construct the groups acting absolutely irreducibly on $\RR^8$ from the ones acting on $\RR^4$ it is crucial that we restrict ourselves to numbers $a$ which are products of prime numbers of the form {$1 \mod 4$}. We quote some number theoretic results first that eventually deliver square roots of $-1$ in suitable congruences. The first and easiest result, which is proved using Wilson's theorem, can be found for example in \autocite{hardy1968an}. For a more thorough historical discussion of this question see \autocite{gauss1870disquisitiones}.

\begin{prop}
\label{propaprime}
Let $a$ be a prime number of the form $a= 1 \mod 4$. Then there exists $\rho \in \NN$ such that
\[ \rho^2 = -1 \mod a.\]
\end{prop}

\noindent
The next step is to apply a method based on Hensel's lemma (see \autocite{eisenbud1995commutative} or \autocite{milne2006elliptic} for the formulation that is used here) that provides the same result for prime powers.

\begin{prop}
\label{propaprimepower}
Let $\tilde{a}$ be a prime number of the form $\tilde{a}= 1 \mod 4$ and $a= \tilde{a}^s$ for some $s\in \NN$. Then there exists $\rho \in \NN$ such that
\[ \rho^2 = -1 \mod a.\]
\end{prop}
\begin{proof}
	We use Hensel's lemma in the formulation given in \autocite{milne2006elliptic} with the polynomial
	\[ f(X) = X^2+1 .\]
	Performing an induction we obtain zeros in congruences of arbitrary powers of $a$, since we have at least the zero modulo $a$ from Proposition \ref{propaprime}.
\end{proof}

\noindent
Now we can apply the Chinese remainder theorem (see for example \autocite{eisenbud1995commutative}) to obtain the result for arbitrary products of prime powers. Hence this completes the proof for \autoref{prime}.

%% file: lseb_gr_v2.tex

\section{Families of groups}
\label{ssma_gr}

\subsection[Representation on \texorpdfstring{$\RR^4$}{R4}]{Representation on $\boldsymbol{\RR^4}$}
In a first step towards the proof of the results on the $4$-dimensional representation we investigate the structure of the groups \Hab as defined before (see \eqref{H_ab}). Note that the generators are subject to several relations which we summarize in the following lemma.

\begin{lemma}
\label{smallrel}
The trivial relations are
\begin{align*}
[e_a,1][1,e_b] &= [1,e_b][e_a,1], &\quad [e_a,1][1,\jj] &= [1,\jj][e_a,1], &\quad [1,e_b][\jj,1] &= [\jj,1][1,e_b], \\
[1,\jj][\jj,1] &= [\jj,1][1,\jj], &\quad [1,\jj]^4 &= [1,1], &\quad [\jj,1]^4 &= [1,1]. 
\end{align*}
The fact that $e_p \cdot \jj = \jj \cdot \bar{e_p} = \jj \cdot e_p^{-1}$ yields
\[ [e_a,1][\jj,1] = [\jj,1][e_a,1]^{-1} = [\jj,1][e_a,1]^{2a-1}, \quad [1,e_b][1,\jj] = [1,\jj][1,e_b]^{-1} = [1,\jj][1,e_b]^{2b-1}. \]
From the identification $[-1,-1] = [1,1]$ we obtain
\[ [e_a,1]^a = [1,\jj]^2, \quad [1,e_b]^b = [1,\jj]^2, \quad [\jj,1]^2 = [1,\jj]^2. \]
\end{lemma}

\noindent
These relations allow us to write every group element $h \in \Hab$ in the form
\begin{equation}
\label{smallgroupel}
h = [e_a,1]^{k_1} [1,e_b]^{k_2} [1,\jj]^{l_1} [\jj,1]^{l_2}
\end{equation}
where $k_1 \in \ZZ/a\ZZ, k_2 \in \ZZ/b\ZZ, l_1 \in \ZZ/4\ZZ$ and $l_2 \in \ZZ/2\ZZ$. We present the proof for \autoref{thmsmallgroups} in the following lemmas.

\begin{lemma}
For each odd $\m \in \NN$ and each decomposition $\m=a \cdot b$ as in \eqref{decm} the group \Hab forms a subgroup of \SO{4} of order $8\m$.
\end{lemma}
\begin{proof}
Comparing \Hab with Table 4.2 in \autocite{conway2003on} and using their notation we find
\[ \Hab = \pm \left[ D_{2a}, D_{2b} \right] \]
where $D_{2n}$ is the \emph{dihedral group} of order $2n$. This group is of order $2 \cdot 2a \cdot 2b = 8\m$. In the notation of \autocite{conway2003on}, the $\pm [\ldots]$ is reflected in a factor $2$ in the group orders.
The order of \Hab can also be derived directly from the representation of group elements in terms of the generators \eqref{smallgroupel}.
\end{proof}

\begin{rem}
The definition of $\Hab$ is symmetric in $a$ and $b$:
\[ \Hab \cong H_{b,a}. \]
However choosing different decompositions for a value of $\m$ leads to groups of the same order which are not necessarily isomorphic. At the end of this section we have listed some concrete examples (see Table \ref{tabsmallgroups}).
\end{rem}

\begin{lemma}
\label{lemsmallsubg}
Let $b<b'$ where $b'$ is odd and relatively prime to $a$. If $b$ divides $b'$, \Hab is a subgroup of $H_{a,b'}$, i.e.
 \[\Hab \le H_{a,b'}. \]
\end{lemma}
\begin{proof}
Let $b$ and $b'$ be as assumed above. There exists $q \in \NN$ with $b' = bq$. Then
\[ e_{b'}^q = e^{\frac{\pi \ii q}{b'}} = e^{\frac{\pi \ii}{b}} = e_b. \]
So we obtain $[1, e_b] \in H_{a,b'}$ and therefore
\[\Hab \le H_{a,b'}. \]
\end{proof}

\begin{rem}
The same result holds for the parameter $a$. The proof is completely analog to the one of the previous lemma.
\end{rem}

\noindent
In the next step we consider the action of \Hab on $\HH$ (see \eqref{smallrep}). 
To prove absolute irreducibility of the representation we follow the strategy of \autocite{lauterbach2010do} from where we use Lemma 6.2 (for the necessary background on representation theory see Chapter 4 in \autocite{chossat2000methods}).

\begin{lemma}
The action of \Hab on $\HH$ is absolutely irreducible.
\end{lemma}
\begin{proof}
We want to use the two two-to-one maps from the ordered pairs of unitary quaternions $\Q \times \Q$ to $\SO{4}$ and $\OO{4} \setminus \SO{4}$ respectively to find linear maps that commute with the group action. Let $[l,r] \in \Q \times \Q$ commute with \Hab. Consider the group element $[1,\jj]$. If $[l,r]$ commutes with $[1,\jj]$, then $r$ commutes with $\jj$. This yields $r=r_1 + r_2 \jj$ with $r_1,r_2 \in \RR$. Furthermore $[1,e_b] = \left[1, \cos \left(\pi/b \right) + \sin \left(\pi/b \right) \ii \right] \in \Hab$ and the relation $r \cdot e_b = e_b \cdot r$ yields $r_2 = 0$, since $\sin \left(\pi/b \right) \ne 0$. Therefore $r \in \RR$ and as $r$ is a unitary quaternion this gives $r= \pm 1$. Performing the same calculations for $l$ using the elements $[\jj,1]$ and $[e_a,1]$, we obtain $l = \pm 1$ as well. All pairs of unitary quaternions that commute with \Hab are $[\pm 1, \pm 1]$. Application of the two-to-one maps from $\Q \times \Q$ to $\SO{4}$ and $\OO{4} \setminus \SO{4}$ respectively yields that the only elements of $\OO{4}$ commuting with the group action are $\pm \mathbbm{1}$. Lemma 6.2 from \autocite{lauterbach2010do} implies absolute irreducibility.
\end{proof}

\noindent
In the following lemmas we investigate the isotropy of the action of \Hab on $\HH$. Using Lemmas 6.3 and 6.4 from \autocite{lauterbach2010do} as well as the relations of generating elements (\autoref{smallrel}) we may prove:

\begin{lemma}
\label{smalliso1}
Let $h = [e_a,1]^{k_1} [1,e_b]^{k_2} [1,\jj]^{l_1} [\jj,1]^{l_2} \in \Hab$ as in \eqref{smallgroupel} with $k_1 \in \rkr{a}, k_2 \in \rkr{b}$ as well as $l_1 \in \lbrace 1,3 \rbrace$ and $l_2=1$. Then
\begin{enumerate}[label={(\arabic*)}]
	\item $h$ fixes a $2$-dimensional subspace of $\HH$;
	\item $h$ is of order $2$;
	\item For $l_1=1$ the fixed point space of $h$ is
		\[ \left\langle \begin{pmatrix}
		\cos \left(\frac{1}{2} \left(\frac{k_1}{a} - \frac{k_2}{b}\right)\pi \right) \\
		\sin \left(\frac{1}{2} \left(\frac{k_1}{a} - \frac{k_2}{b}\right)\pi \right) \\
		0 \\
		0
		\end{pmatrix}, \begin{pmatrix}
		0 \\
		0 \\
		\cos \left( \frac{1}{2} \left(\frac{k_1}{a} + \frac{k_2}{b}\right)\pi \right) \\
		\sin \left( \frac{1}{2} \left(\frac{k_1}{a} + \frac{k_2}{b}\right)\pi \right)
		\end{pmatrix} \right\rangle. \]
		For $l_1=3$ the fixed point space of $h$ is
		\[ \left\langle \begin{pmatrix}
		\cos \left(\frac{1}{2} \left(\frac{k_1}{a} - \frac{k_2}{b} +1 \right)\pi \right) \\
		\sin \left(\frac{1}{2} \left(\frac{k_1}{a} - \frac{k_2}{b} +1 \right)\pi \right) \\
		0 \\
		0
		\end{pmatrix}, \begin{pmatrix}
		0 \\
		0 \\
		\cos \left( \frac{1}{2} \left(\frac{k_1}{a} + \frac{k_2}{b} +1 \right)\pi \right) \\
		\sin \left( \frac{1}{2} \left(\frac{k_1}{a} + \frac{k_2}{b} +1 \right)\pi \right)
		\end{pmatrix} \right\rangle. \]
\end{enumerate}
\end{lemma}

\noindent
The previous lemma describes restrictions on the exponents in the representation (\ref{smallgroupel}) that guarantee nontrivial isotropy. These are in fact all the elements with nontrivial fixed point spaces. Using Lemmas 6.5 and 6.6 in \autocite{lauterbach2010do}, we see that elements which do not meet these restrictions fix the origin only.
\begin{lemma}
\label{smalliso2}
Let $h = [e_a,1]^{k_1} [1,e_b]^{k_2} [1,\jj]^{l_1} [\jj,1]^{l_2} \in \Hab \setminus \lbrace[1,1]\rbrace$ as in \eqref{smallgroupel} with $k_1 \in \rkr{a}, k_2 \in \rkr{b}$ as well as $l_1 \not\in \lbrace 1,3 \rbrace$ or $l_2\ne1$. Then $h$ fixes only $0 \in \HH$.
\end{lemma}

The form of group elements that have nontrivial fixed points from \autoref{smalliso1} guarantees that nontrivial isotropy subgroups can only contain one such element. The product of two different elements with fixed point space can not fix a point besides $0$.

\begin{lemma}
\label{smalliso3}
The nontrivial isotropy subgroups of \Hab are generated by precisely one group element.
\end{lemma}

\noindent
To shorten notation we want to name the two types of isotropy subgroups as follows for the rest of this subsection
\begin{align*}
K	&= \langle h \rangle = \langle [e_a,1]^{k_1} [1,e_b]^{k_2} [1,\jj] [\jj, 1] \rangle \\
K'	&= \langle h' \rangle = \langle [e_a,1]^{k_1} [1,e_b]^{k_2} [1,\jj]^3 [\jj, 1] \rangle.
\end{align*}

\begin{lemma}
\label{smalliso4}
The isotropy groups $K$ and $K'$ are conjugate either to $\langle [1, \jj][\jj, 1] \rangle = \langle [\jj, \jj] \rangle$ or to $\langle [1, \jj]^3 [\jj,1] \rangle = \langle -[\jj, \jj] \rangle$.
\end{lemma}
\begin{proof}
This can be calculated directly using the relations on the generating elements of $K = \langle h \rangle$ and $H' = \langle h' \rangle$ and the fact that $a$ and $b$ are odd. For $K$ we obtain
\begin{align*}
h	&= \left( [1, \jj]^2 \left( [e_a,1]^{\frac{a+1}{2}} \right)^2 \right)^{k_1} \left( [1,\jj]^2 \left( [1,e_b]^{\frac{b+1}{2}} \right)^2 \right)^{k_2} [1,\jj] [\jj, 1] \\
	&= \tilde{h} \left( [1,\jj]^2 \right)^{k_1 + k_2} [1,\jj] [\jj, 1] \tilde{h}^{-1}
\end{align*}
with 
\[ \tilde{h} = [e_a,1]^{k_1 \frac{a+1}{2}} [1,e_b]^{k_2 \frac{b+1}{2}} .\]
Since
\[ \left( [1,\jj]^2 \right)^{k_1 + k_2} = \begin{cases}
[1,1] \quad	&\text{for } k_1+k_2 \text{ even}, \\
[1, \jj]^2 \quad	&\text{for } k_1+k_2 \text{ odd},
\end{cases} \]
this yields the claim for $K$. The proof for $K'$ is completely alike.
\end{proof}

\noindent
The previous lemma completes the proof for \autoref{thmsmallgroups} on the groups \Hab and their $4$-dimensional representations. The proofs for the Lie group structure are straightforward from the corresponding properties of the finite groups, where the $2$-dimensional Lie group (\autoref{thmsmallgroups3}) arises in the same manner as the one-dimensional Lie group (\autoref{thmsmallgroups2}).

\begin{proof}[Proof of \autoref{thmsmallgroups2}]
\begin{enumerate}[label={(\arabic*)}]
\item The closure of the union of the groups \Hab over all suitable $b$ is the smallest group that contains the elements $\left\lbrace [e_a,1], [1,\jj], [\jj,1], [1,e^{\psi\ii}] \mid \phi \in S^1 \right\rbrace$ (note that $[1,e^{(\pi + \psi)\ii}] = [-1,e^{\psi\ii}]$ for $\psi \in [0, \pi)$). Write it as follows
\[ \Ha = \left\langle [e_a,1], [1,\jj], [\jj,1], [1,e^{\psi\ii}] \mid \psi \in S^1 \right\rangle. \]
This is a compact $1$-dimensional Lie group. It contains an element $[e_a,e_a]$ which fixes $\langle 1, \ii \rangle$ as a real subspace of $\HH$. Furthermore $[\jj, \jj] \in \Hab$ fixes the real subspace $\langle 1, \jj \rangle$. Thus the subgroup generated by these two elements $\langle [e_a,e_a], [\jj, \jj] \rangle$ fixes the one-dimensional real subspace $\langle 1 \rangle \subset \HH$.
\item Write \Ha and $\mathbf{H}_{a'}$ in terms of generators. Then for $a,a' \in 2\NN+1$ with $a < a'$ and $a$ divides $a'$ we obtain $[e_a,1] \in \mathbf{H}_{a'}$. The claim follows as in the proof of Theorem \ref{thmsmallgroups}.
\end{enumerate}
\end{proof}

For the construction of the groups acting on $\RR^8$ we need the matrix representation of the groups \Hab with respect to the standard basis of $\RR^4$. It can be calculated directly via applying the generators to the basis elements. One obtains
\begin{equation}
\label{MatrixRep}
\begin{alignedat}{3}
c &=
\begin{pmatrix}
\rr\left( \frac{\pi}{a} \right)	& 0 \\
0								& \rr\left( \frac{\pi}{a} \right)
\end{pmatrix}& &\quad
d& &=
\begin{pmatrix}
\rr\left( -\frac{\pi}{b} \right)	& 0 \\
0									& \rr\left( \frac{\pi}{b} \right)
\end{pmatrix} \\
q &= \begin{pmatrix}
0 & 0 & 1 & 0 \\
0 & 0 & 0 & 1 \\
-1 & 0 & 0 & 0 \\
0 & -1 & 0 & 0
\end{pmatrix}& &\quad
s& &= \begin{pmatrix}
0 & 0 & -1 & 0 \\
0 & 0 & 0 & 1 \\
1 & 0 & 0 & 0 \\
0 & -1 & 0 & 0
\end{pmatrix}
\end{alignedat}
\end{equation}where $r(\psi)$ is again the $2$-dimensional rotation matrix by an angle $\psi$. One can see that the first two elements correspond to blockwise rotations in two coordinates respectively. 

\begin{lemma}
\label{smallmatrel}
The matrix generators are subject to the same relations as the corresponding pairs of quaternions (see Lemma \ref{smallrel}): $cd = dc, cq = qc, ds = sd, qs = sq, q^4 = \mathbbm{1}_4, s^4 = \mathbbm{1}_4, cs = sc^{-1} = sc^{2a-1}, dq = qd^{-1} = qd^{2b-1}, c^a = q^2, d^b = q^2, s^2 = q^2$.
\end{lemma}

The group algebra software GAP \cite{gap2014} allows to check some of the stated results for low group orders. Among other things, GAP provides a classification scheme for small groups. The GAP-identifiers are composed of two integers. The first one is the group order and the second one enumerates the isomorphism classes of groups of the given order. Here we present the identifiers of the first few groups within our classification which are relevant for our subsequent analysis.
\begin{table}[!ht]
\begin{center}
\resizebox{\textwidth}{!}{%
\begin{tabular}{|c|c|c|c|c|c|c|c|c|} \hline
$\m$		& 15 		& 35 		& 39 		& 45 		& 51 		& 55		& \emph{65}	& \emph{65} \\ \hline
$(a,b)$ & (5,3) 	& (5,7) 	& (13,3) 	& (5,9) 	& (17,3) 	& (5,11)	& \emph{(5,13)}	& \emph{(13,5)} \\ \hline
GAP-id. 	& [120, 10]	& [280, 9] 	& [312, 17]	& [360, 9] 	& [408, 9] 	& [440, 19]	& \emph{[520,13]}	& \emph{[520,13]} \\ \hline
\end{tabular}}
\end{center}
\captionsetup{width=.95\textwidth}
\caption{GAP-identifiers of \Hab for small values of $\m$. GAP identifies groups by their order in the first position and an enumeration of the isomorphism classes in the second position. Note the symmetry in $a$ and $b$ in the case $\m = 65$ where both factors are in $\A$. In cases like this we have multiple choices for $a$. Bear in mind that there are more complicated cases in which we have more than two choices for $a$ and $b$ for the same value of $\m$. In these cases, a change of the parameters does not necessarily lead to the same groups. However, the smallest groups for which this occurs are already far beyond the reach of the SmallGroup library from GAP.}
\label{tabsmallgroups}
\end{table}

\subsection[Representation on \texorpdfstring{$\RR^8$}{R8}]{Representation on $\boldsymbol{\RR^8}$}
\label{subsecgroups}
We want to investigate the groups \Gab in a similar manner as the groups \Hab before. First of all we can calculate relations on the generators. The relations on $C,D,Q$ and $S$ are the same as for $c,d,q$ and $s$ because of the blockdiagonal structure. The relations containing $V$ can be calculated using the ones for the small matrices (see Lemma \ref{smallmatrel}).

\begin{lemma}
\label{rel}
The blockdiagonal generators of \Hab are subject to the following relations: $CD = DC, CQ = QC, DS = SD, QS = SQ, Q^4 = \mathbbm{1}_8, S^4 = \mathbbm{1}_8, CS = SC^{-1} = SC^{2a-1}, DQ = QD^{-1} = QD^{2b-1}, C^a = Q^2, D^b = Q^2, S^2 = Q^2$.
Adding $V$ yields
\begin{align*}
VC	&= C^\rho V &\quad
CV	&= VC^{-\rho} = VC^{2a - \rho} &\quad
VD	&= DV \\
VQ	&= Q^3 V &\quad
VS	&= SV &\quad
V^8	&= \mathbbm{1}_8 &\quad
V^2	&= S.
\end{align*}
\end{lemma}

\begin{rem}
This is the point where the fact that $\rho$ is odd becomes important. If it were even $C^a$ would not be equal to $Q^2=-\mathbbm{1}_8$ but 
\[ C^a = \begin{pmatrix}
- \mathbbm{1}_4 & 0 \\
0				& \mathbbm{1}_4
\end{pmatrix}. \]
\end{rem}

\noindent
Just as before this allows us to write every element $g \in \Gab$ in the form
\begin{equation}
\label{groupel}
g = C^{k_1} D^{k_2} Q^{l_1} S^{l_2} V^m
\end{equation}
with $k_1 \in \rkr{a}, k_2 \in \rkr{b}, l_1 \in \rkr{4}, l_2 \in \rkr{2}$ and $m \in \rkr{2}$. Using the calculations for the groups \Gab we may state similar results on the structure and isotropy of the groups \Hab.

\begin{lemma}
For each $a \in \A$ and $b \in 2\NN+1$ with $\gcd(a,b)=1$ the group \Gab forms a subgroup of \OO{8} of order $16\m$ where $\m = a \cdot b$.
\end{lemma}
\begin{proof}
The elements $C,D,Q$ and $S$ generate a group that is isomorphic to \Hab. Addition of $V$ to the set of generators gives two copies of this group. Therefore $ \left| \Gab \right| = 2 \left| \Hab \right| = 16\m$. Once again this can be calculated directly from the form of the group elements \eqref{groupel}.
\end{proof}

\begin{lemma}
\label{lemsubg}
Let $b<b'$ where $b'$ is odd and relatively prime to $a$. If $b$ divides $b'$, \Gab is a subgroup of $G_{a,b'}$, i.e.
\[ \Gab \le G_{a,b'}. \]
\end{lemma}
\begin{proof}
This follows directly from the second statement in Theorem \ref{thmsmallgroups}.
\end{proof}

\begin{rem}
In contrast to the $4$-dimensional case, we do not obtain the same result for the parameter $a$, which is due to the power $\rho$ in the definition of the generating element $C$. This power is not necessarily the same for $a$ and $a'$ when $a$ divides $a'$. However in this case \Gab is isomorphic to a subgroup of $G_{a',b}$.
\end{rem}

\begin{lemma}
The natural $8$-dimensional representation of \Gab is absolutely irreducible.
\end{lemma}
\begin{proof}
Let $L \colon \RR^8 \to \RR^8$ be a linear map in matrix representation that commutes with the group action of \Gab. We write $L$ in form of a block matrix
\[ L = \begin{pmatrix}
L_{1,1}	& L_{1,2} \\
L_{2,1}	& L_{2,2}
\end{pmatrix} \]
where $L_{i,j} \colon \RR^4 \to \RR^4$ for $i,j \in \lbrace 1,2 \rbrace$. In a first step we want to show that $L_{1,2}=L_{2,1}=0$. Then we can use absolute irreducibility of the $4$-dimensional representation of \Hab to prove the claim. Using  the commutativity assumption and the structure of the generating matrices we obtain
\begin{equation*}
\begin{alignedat}{6}
L_{1,2}qs		&=& &-qsL_{1,2} &	&\quad \text{from} \quad& &LQ& &=& &QL \text{ and } LS = SL, \\
L_{1,2}c^\rho	&=& &cL_{1,2} &	&\quad \text{from} \quad& &LC& &=& &CL, \\
L_{1,2}s 		&=& &L_{2,1}&		&\quad \text{from} \quad& &LV& &=& &VL.
\end{alignedat}
\end{equation*}
The first relation yields that $L_{1,2}$ is of the form
\[ L_{1,2} = \begin{pmatrix}
0 & * & 0 & * \\
* & 0 & * & 0 \\
0 & * & 0 & * \\
* & 0 & * & 0
\end{pmatrix}. \]
We want to apply the second relation and remember that $c$ is the representation matrix of $[e_a,1]$ on $\HH$. Thus we calculate the power of $c$ to be
\[ c^\rho = 
\begin{pmatrix}
\rr\left( \frac{\rho \pi}{a} \right)	& 0 \\
0										& \rr\left( \frac{\rho \pi}{a} \right)
\end{pmatrix}
. \]
Note that the entries of this matrix contain the real and imaginary part of $e_a^\rho$:
\begin{equation*}
\cos \left(\frac{\rho \pi}{a}\right) = \Re \left(e_a^\rho\right), \quad
\sin \left(\frac{\rho \pi}{a}\right) = \Im \left(e_a^\rho\right).
\end{equation*}
Now we make use of the special choice of the power $\rho$ to prove that these can not match the real and imaginary part of $e_a$. Since $\rho \in \lbrace 0, \ldots, a-1 \rbrace$, we obtain that the only chance for $\Re \left(e_a^\rho\right) = \Re \left(e_a\right)$ is for $\rho = 1$ or $\rho = 2a-1$ which both contradict the fact that $\rho^2 = -1 \mod a$. Considering the imaginary parts, we obtain that the only possibility for $\Im \left(e_a^\rho\right) = \Im \left(e_a\right)$ is if $\rho = 1$ or $\rho = a-1$. Once again this contradicts the choice of $\rho$. Therefore 
\begin{equation*}
\cos \left(\frac{\rho \pi}{a}\right) \ne \cos \left(\frac{\pi}{a}\right) \quad \text{and} \quad
\sin \left(\frac{\rho \pi}{a}\right) \ne \sin \left(\frac{\pi}{a}\right).
\end{equation*}
Omitting the details, this allows us to compute that the remaining entries of $L_{1,2}$ are zero as well. Together with the last relation this yields $L_{1,2}=L_{2,1}=0$.

Therefore we obtain two linear maps $L_{i,i} \colon \RR^4 \to \RR^4$ for $i=1,2$ that commute with the action of \Hab. From the absolute irreducibility of this action we know that $L$ is of the form
\[ L = \begin{pmatrix}
\gamma \mathbbm{1}_4 	& 0 \\
0						& \delta \mathbbm{1}_4
\end{pmatrix} \]
with $\gamma, \delta \in \RR$. Commutation with $V$ yields $\gamma = \delta$.
\end{proof}

In the next step we investigate the isotropy of the $8$-dimensional representation of \Gab. Note that the corresponding results on \Hab mostly rely on the relations of the generating elements. Hence they can be adapted almost directly.

\begin{lemma}
Let $g \in \Gab \setminus \lbrace \mathbbm{1}_8 \rbrace$ be written in the form \eqref{groupel}. Then $g$ fixes a point $x \in \RR^8 \setminus \lbrace 0 \rbrace$ if and only if $l_1 \in \lbrace 1,3 \rbrace, l_2 = 1$ and $m=0$.
\end{lemma}
\begin{proof}
For $m=0$ the claim follows directly from Lemmas \ref{smalliso1} and \ref{smalliso2} since the other elements keep the two \Hab-blocks intact. Therefore we consider elements of the form \eqref{groupel} with $m=1$:
\[ g = C^{k_1} D^{k_2} Q^{l_1} S^{l_2} V. \]
Suppose $x = (\zeta, \eta)$ with $\zeta,\eta \in \RR^4$ such that $gx = x$. Using the structure of the generating matrices, this yields
\[ gx = C^{k_1} D^{k_2} Q^{l_1} S^{l_2} \left(\eta, s\zeta \right) = (\zeta, \eta). \]
Since $C,D,S$ and $Q$ keep the block structure intact, we may split this into two equations:
\begin{align*}
c^{k_1} d^{k_2} q^{l_1} s^{l_2} \eta 				&= \zeta \\
c^{\rho k_1} d^{k_2} q^{3 l_1} s^{l_2 + 1} \zeta	&= \eta.
\end{align*}
Inserting the second equation in the first one, we obtain
\[ \left(c^{k_1} d^{k_2} q^{l_1} s^{l_2} \right) \left( c^{\rho k_1} d^{k_2} q^{3 l_1} s^{l_2 + 1} \right) \zeta = \zeta. \]
Using the relations on the matrix representation of the generators of \Hab (Lemma \ref{smallmatrel}), we may then calculate
\[ c^{k_1} d^{k_2} q^{l_1} s^{l_2} c^{\rho k_1} d^{k_2} q^{3 l_1} s^{l_2 + 1} = c^{k_1+ (-1)^{l_2}\rho k_1} d^{k_2 + (-1)^{l_1} k_2} q^{2 l_2} s. \]
Since the power of $q$ is even, Lemma \ref{smalliso2} yields $\zeta = 0$. Inserting this in the second equation gives $\eta = 0$ which completes the proof.
\end{proof}

\begin{rem}
Note that we can use the formulas to compute basis elements of the fixed point spaces (\autoref{smalliso1}) in the case $m=0$ as well. We only have to take the powers of $c$ and $q$ in the second block of the matrices $C$ and $Q$ into account. The fixed point spaces are obviously $4$-dimensional.
\end{rem}

\noindent
Concerning the isotropy subgroups of \Gab we obtain the same result as in \autoref{smalliso3} from the fact that $\langle C,D,Q,S \rangle$ is isomorphic to \Hab:

\begin{lemma}
The nontrivial isotropy subgroups of \Gab are generated by precisely one group element.
\end{lemma}

Once more we want to make use of shorter notations. We therefore write
\begin{align*}
K	&= \langle g \rangle = \left\langle C^{k_1} D^{k_2} Q S \right\rangle \\
K'	&= \langle g' \rangle = \left\langle C^{k_1} D^{k_2} Q^3 S \right\rangle
\end{align*}
for the two types of nontrivial isotropy subgroups for the rest of this subsection. Using the element $V$, we may now show, that in the $8$-dimensional case we obtain only one isotropy type:

\begin{lemma}
All nontrivial isotropy subgroups of \Gab are conjugate to $\langle QS \rangle$.
\end{lemma}
\begin{proof}
All nontrivial isotropy subgroups are generated by either $g$ or $g'$ which both do not contain a factor $V$. We may therefore use Lemma \ref{smalliso4} and the fact that $C,D,Q$ and $S$ are subject to the same relations as $[e_a,1], [1,e_b], [1, \jj]$ and $[\jj, 1]$. This yields that every nontrivial isotropy subgroup is conjugate to either $\langle QS \rangle$ or $\langle Q^3S \rangle$. These two subgroups are conjugate by $V$:
\[ VQSV^{-1} = Q^3SVV^{-1} = Q^3S. \]
Thus all nontrivial isotropy subgroups are conjugate to $\langle QS \rangle$.
\end{proof}

\noindent
This completes the proof for \autoref{thmgroups}. Similarly to the $4$-dimensional case these considerations leave the results on the Lie group structure straightforward and we may state the proof of \autoref{thmgroups2}:

\begin{proof}[Proof of \autoref{thmgroups2}]
The claim follows in the same way as in the proof for the groups \Ha. Let $\lbrace \xi_1, \ldots, \xi_8 \rbrace$ be the standard basis of $\RR^8$. The Lie group \Ga contains the elements $C$ and $D\left(\pi/a\right)$ and their product $C D\left(\pi/a\right)$ fixes the subspace $\left\langle \xi_1, \xi_2 \right\rangle$. Furthermore $QS$ fixes the subspace $\left\langle \xi_1, \xi_3, \xi_6, \xi_8 \right\rangle$. Thus the subgroup $\left\langle C D\left(\pi/a\right), QS \right\rangle$ generated by these two elements fixes the subspace $\left\langle \xi_1 \right\rangle$.
\end{proof}

\begin{rem}
As mentioned before the subgroup relation for the one-dimensional Lie groups as in \autoref{thmsmallgroups2} does not hold because of the exponent $\rho$ in the construction of the matrix $C$. Furthermore we do not obtain Lie groups of dimension $2$ when considering the closure of the union of the one-dimensional Lie groups \Ga over all $a \in \A$. The reason for this structural difference lies in the power $\rho$ as well. It is a nonconstant natural number depending on the angle which is a rational multiple of $\pi$. As such it has no smooth -- more precisely, not even a continuous -- continuation on all angles $\phi \in S^1$ and therefore prevents a smooth structure on the matrices $C$ for all angles. It is unknown whether there exist $2$-dimensional Lie groups containing all the groups $\Ga$.
\end{rem}

\noindent
We provide the GAP-identifiers for the first groups \Gab (compare to Table \ref{tabsmallgroups}):
\begin{table}[!ht]
\begin{center}
\resizebox{\textwidth}{!}{%
\begin{tabular}{|c|c|c|c|c|c|c|c|c|} \hline
$\m$ & 15 & 35 & 39 & 45 & 51 & 55 & \emph{65} & \emph{65} \\ \hline
$(a,b)$ & (5,3) & (5,7) & (13,3) & (5,9) & (17,3) & (5,11) & \emph{(5,13)} & \emph{(13,5)} \\ \hline
GAP-id. & [240, 101] & [560, 94] & [624, 130] & [720, 98] & [816, 97] & [880, 130] & \emph{[1040,105]} & \emph{[1040,112]} \\ \hline
\end{tabular}}
\end{center}
\captionsetup{width=.95\textwidth}
\caption{GAP-identifiers of \Gab for small values of $\m$. Note that for the groups \Gab the symmetry in $a,b$ is broken in the case of both factors being in $\A$. This is due to the different construction of the matrix $C$ from $c$.}
\label{tabgroups}
\end{table}

%% file: lseb_es_v2.tex

\section{Equivariant structure for \texorpdfstring{$\boldsymbol{\GG_a}$}{Ga}}
\label{ssma_es}
Since we are interested in bifurcation problems on $\RR^8$ with \Gab -symmetry for suitable $a$ and $b$, we have to investigate smooth \Gab -equivariant maps on $\RR^8$. Using methods from character and invariant theory, we are able to compute dimensions of spaces of \emph{homogeneous equivariant polynomial maps} for up to third degree. Then we determine the generating functions for the corresponding spaces. This allows us to gain insight in the general bifurcation behaviour of equations with \Gab -symmetry which we will investigate further in the next section.

For a group $\Gamma$ acting on the real space $W$ define its \emph{character} as follows
\[ \chi \colon \Gamma \to \RR, \quad g \mapsto \tr{g} \]
and denote the space of smooth $\Gamma$-equivariant maps by $C_\Gamma^\infty \left( W, W \right)$. It is well known that the symmetric functions form a module which contains the equivariant polynomials as a dense subset (see for example \autocite{chossat2000methods} or \autocite{field2007dynamics}). The space of homogeneous equivariant polynomial maps of degree $d$ shall be denoted by $P_\Gamma^d (W,W)$. To gather information about the equivariant structure of a given representation one often looks at the so called \emph{Molien series}, a formal power series that carries information about dimensions of these spaces. It is defined as follows
\[ \sum_{d = 0}^{\infty} R_d z^d\]
where $R_d = \dim P_\Gamma^d (W,W)$ is the number of linearly independent equivariant polynomial maps of degree $d$ to which we refer as \emph{Molien coefficients}. In a similar way we consider invariant polynomials from the represention space into the real numbers. These are in a close relationship to the equivariant polynomial maps. We denote the space of invariant homogeneous polynomials of degree $d$ by $\Pi_\Gamma^d (W)$. Then we obtain for example that for every $p\in \Pi_\Gamma^d(W)$ the gradient $\nabla p$ is an equivariant polynomial map: $\nabla p \in P_{\Gamma}^{d-1}\left(W,W\right)$. For more details on this matter and the connection between invariant and equivariant polynomials see \autocite{chossat2000methods}. The corresponding formal power series
\[ \sum_{d = 0}^{\infty} r_d z^d\]
with $r_d = \dim \Pi_\Gamma^d(W)$ is called Molien series as well. Molien's theorem states a way to calculate these formal power series but it is often difficult to do so. That is why we use a slightly different approach.

\subsection{Computation of Molien coefficients}
\label{secsatt}
We are especially interested in the equivariant structure for low degree polynomial maps. \autocite{sattinger1979group} proves a formula by which we can calculate the $R_d$ for a single $d$ without having to deal with the Molien series. This formula also follows from the results in \autocite{zhilinskii1989tensor}. Although it is impractical for large values of $d$ it is very helpful in the cases which we consider. For $g\in \Gamma$ \citeauthor{sattinger1979group} defines the quantity
\[ \chi_{(d)}(g) = \sum_{\sum_{k=1}^{d} k\cdot i_k = d} \frac{\chi^{i_1}(g) \cdots \chi^{i_d} \left( g^d \right)}{1^{i_1}i_1 ! 2^{i_2}i_2 ! \cdots d^{i_d} i_d !} \]
and obtains
\begin{equation}
\label{eqsat2}
R_d = \int_{\Gamma} \chi_{(d)}(g) \chi(g) dg.
\end{equation}
Note that in the case of a finite group the integral becomes a normed sum. To compute single Molien coefficients $r_d$ for the invariant polynomials there exists a similar formula that can easily be derived from the calculations in \autocite{zhilinskii1989tensor}:
\[ r_d = \int_{\Gamma} \chi_{(d)}(g) dg .\]

For the bifurcation analysis we are only interested in the equivariant structure and we will see later that we only need the data for degrees up to $d=3$ (see Section \ref{ssma_mr}). For these cases we can apply formula \eqref{eqsat2} with reasonable effort. In the case of an absolutely irreducible representation, the only linear maps commuting with the group action are multiples of the identity, therefore we immediately obtain $R_1 = 1$. Furthermore $\chi_{(2)}$ reads
\[ \chi_{(2)} (g) = \frac{1}{2} \left( \chi(g^2) + \chi^2 (g)\right). \]
To calculate $\chi_{(3)}$ using $ i_1 + 2i_2 + 3i_3 = 3$ we have the choices $(3,0,0), (1,1,0)$ and $(0,0,1)$ for $(i_1, i_2, i_3)$. Therefore we get
\[ \chi_{(3)} (g) = \frac{1}{3!}\chi^3(g) + \frac{1}{2} \chi(g)\chi(g^2) + \frac{1}{3} \chi(g^3). \]

We want to use formula \eqref{eqsat2} to calculate $R_2$ and $R_3$ for the groups \Gab. Let $a \in \A$ and $b \in 2\NN+1$ with $\gcd(a,b)=1$ be natural numbers as before. In a first step we investigate the character $\chi: \Gab \to \RR$ for an arbitrary element $g \in \Gab$. It is very useful to notice that the character is a class function, i.e. it is invariant under conjugation. We have seen that $g$ can be written in the form
\[ g = C^{k_1} D^{k_2} Q^{l_1} S^{l_2} V^m \]
with $k_1 \in \rkr{a}, k_2 \in \rkr{b}, l_1 \in \rkr{4}, l_2 \in \rkr{2}$ and $m \in \rkr{2}$ (see \eqref{groupel}). Recall how \Gab is constructed from \Hab and note that $C, D, Q$ and $S$ keep the two \Hab -blocks intact. This yields that every element $g$ with $m=1$ is of the form
\[ g = \begin{pmatrix}
0	& * \\
*	& 0
\end{pmatrix} \]
and therefore $\chi (g) = 0$. Hence we may restrict to the case $m=0$:
\[ g = C^{k_1} D^{k_2} Q^{l_1} S^{l_2}. \]
These elements are of the form
\begin{equation}
\label{bigblock}
g = \begin{pmatrix}
h	& 0 \\
0	& h'
\end{pmatrix}
\end{equation}
with $h, h' \in \Hab$ (in matrix representation) and for their character we obtain
\[ \chi (g) = \chi_4 (h) + \chi_4 (h') \]
where $\chi_4 \colon \Hab \to \RR$ denotes the character of the $4$-dimensional representation of \Hab. Investigating this character provides us with the needed result. To do so we make use of both the biquaternionic and the matrix representation of \Hab.

Similar to the form of the group element $g \in \Gab$ we may characterize
\[ h = c^{k_1} d^{k_2} q^{l_1} s^{l_2} \]
for $h \in \Hab$ (see \eqref{smallgroupel}). In a similar manner as before we obtain
\[ h = \begin{pmatrix}
0	& * \\
*	& 0
\end{pmatrix} \]
if $l_1 + l_2$ is odd and therefore $\chi_4 (h) = 0$ in this case. Consider $l_1 = 3$ and $l_2 = 1$. In Lemma \ref{smalliso4} we have seen that elements of this form are conjugate to either $qs = \diag{(1,-1,1,-1)}$ or $-qs$ and therefore $\chi_4 (h) = 0$. For $l_1 = 1$ and $l_2=1$ we have
\[ h = c^{k_1} d^{k_2} q s = - c^{k_1} d^{k_2} q^{3} s \]
and by linearity of the character $\chi_4 (h) = 0$ as well. The remaining two cases are $l_1 \in \lbrace 0,2 \rbrace$ and $l_2 = 0$. Once more note that
\[ c^{k_1} d^{k_2} q^2 = - c^{k_1} d^{k_2} = -h \]
for $h = c^{k_1} d^{k_2}$ and we may make use of the linearity again. The matrix $h$ corresponds to the group element $\left[ e_a^{k_1}, e_b^{k_2} \right]$ and we compute it to be
\[ h = 
\begin{pmatrix}
\rr\left(\left(\frac{k_1}{a} - \frac{k_2}{b} \right) \pi\right) & 0 \\
0 & \rr\left(\left(\frac{k_1}{a} + \frac{k_2}{b} \right) \pi\right)
\end{pmatrix}
. \]
From now on let
\[ \eta = \frac{k_1}{a} \pi \quad \text{and} \quad \nu = \frac{k_2}{b} \pi. \]
Then we obtain
\[ \chi_4 (h) = 4 \cos \left(\eta \right)\cos \left(\nu \right). \]
Summarizing this yields the only nonzero cases for $l_1 \in \lbrace 0,2 \rbrace$ and $l_2 = 0$ giving
\[ \chi_4 (h) = (-1)^{\frac{l_1}{2}} 4 \cos \left(\eta \right)\cos \left(\nu \right). \]

Returning back to $g \in \Gab$ with $m=0$ and using the block structure \eqref{bigblock} we obtain 
\begin{align*}
h	&= c^{k_1} d^{k_2} q^{l_1} s^{l_2} \\
h'	&= c^{\rho k_1} d^{k_2} q^{3 l_1} s^{l_2}
\end{align*}
for the \Hab -blocks. We investigate the same cases for the powers as before. If $l_1 + l_2$ is odd, then so is $3l_1 + l_2$ and therefore $\chi_4 (h') = 0$ giving $\chi (g) = 0$. If $l_2=1$ and $l_1$ is odd then so is $3l_1$ and in the same manner we obtain $\chi (g) = 0$. For $l_2 = 0$ and $l_1$ even we obtain
\[ 3l_1 = \begin{cases}
0 				&\quad \text{for} \quad l_1 = 0, \\
6 = 2 \mod 4	&\quad \text{for} \quad l_1 = 2.
\end{cases} \]
This yields
\[ \chi (g) = (-1)^{\frac{l_1}{2}} 4 \left( \cos \left(\eta \right) + \cos \left(\rho \eta \right) \right) \cos \left(\nu \right) \]
if $l_1 \in \lbrace 0,2 \rbrace$ and $l_2=0$ and $\chi (g) = 0$ in all other cases.

Knowing the character for every element $g \in \Gab$ allows us to calculate the quantities $\chi_{(d)}$. Note that we only need them for group elements with $\chi (g) \ne 0$ because of the corresponding factor in the dimension formula \eqref{eqsat2}. To perform these calculations for $d=2, 3$ we still need to consider $\chi\left(g^d\right)$. Note that for $l_1 \in \lbrace 0,2 \rbrace$ and $l_2=m=0$ we obtain
\[ g^2 = C^{2 k_1} D^{2 k_2} Q^{2 l_1} = C^{2 k_1} D^{2 k_2}, \]
using the relations on the generating elements, and therefore
\[ \chi (g^2) = 4 \left( \cos \left(2 \eta \right) + \cos \left(2 \rho \eta \right) \right) \cos \left(2 \nu \right). \]
In an analogue way we obtain
\[ g^3 = C^{3 k_1} D^{3 k_2} Q^{l_1} \]
and therefore
\[ \chi (g^3) = (-1)^{\frac{l_1}{2}} 4 \left( \cos \left(3 \eta \right) + \cos \left(3 \rho \eta \right) \right) \cos \left(3 \nu \right). \]
We can then put the parts together to obtain $\chi_{(d)}$ for $d=2,3$ which we use to calculate $R_2$ and $R_3$. The remaining steps are a subtle computation using calculation rules for cosine and the geometric sum formula. The details shall be omitted at this point but can be found in the Appendix (\ref{appsatt}). Performing the calculations we obtain:
\begin{lemma}
	The dimensions $R_d=\dim P_{\Gab}^d\left(\RR^8, \RR^8 \right)$ for $d=1,2,3$ are
	\begin{align*}
	R_1 &= 1, \\
	R_2 &= 0, \\
	R_3 &= \begin{cases}
	8	&\quad \text{for} \quad a=5, \\
	5	&\quad \text{else}.
	\end{cases}
	\end{align*}
\end{lemma}

\subsection[Equivariant maps in the case \texorpdfstring{$a=5$}{a=5} and \texorpdfstring{$b=3$}{b=3}]{Equivariant maps in the case $\boldsymbol{a=5}$ and $\boldsymbol{b=3}$}
\label{subsec_es}

We want to determine the equivariant structure up to third degree for the smallest group we can construct with the method presented in Sections \ref{lseb_mr} and \ref{ssma_gr} which is $G_{5,3}$. The groups in the family $\GG_5$ form a special case in our considerations as we have seen from the calculations of the Molien coefficients. We point out when this is important in a remark at the end of the section. By irreducibility we already know that the only linear equivariants are scalar multiples of the identity. Furthermore we  have no quadratic $G_{5,3}$-symmetric maps on $\RR^8$, since $R_2=0$ and the space of cubic equivariants is $8$-dimensional. 

There are several ways to find equivariant maps of a given degree. \autocite{sattinger1979group} investigates some simple examples. \autocite{lauterbach2010do} describe methods for groups that are constructed in a similar way as the ones we consider using complex polynomials. More general results and computer algebra systems can be found in \autocite{gatermann1999computer, gatermann1991software} and \autocite{gatermann1996grobner}. We have chosen an elementary method to calculate a basis using general homogeneous polynomials and having Maple \cite{maple2012} solve for the coefficients under the assumption of equivariance with respect to the generating matrices. We obtain eight linearly independent polynomial maps $E_1, \ldots ,E_8$ that prove to meet the symmetry condition. They can be found in the Appendix in \autoref{appce}.

\subsection{The general case}
\label{subsec_gc}

We want to use the results for the case $a=5, b=3$ to obtain the full picture for all groups. By construction of the groups \Gab it follows that the only dependence on the parameters $a$ and $b$ is in the matrices $C(a)$ and $D(b)$. A short calculation shows that the vector fields $E_1, \ldots , E_5$ remain equivariant with respect to the matrices $C(a)$ and $D(b)$ with arbitrary $a \in \A$ and $b \in 2\NN+1$ such that $\gcd(a,b)=1$. We may even prove that $E_1, \ldots, E_5$ are equivariant with respect to a matrix $D(\psi)$ that describes an arbitrary angle of rotation $\psi \in S^1$. This gives us the final result on the equivariant structure and hence completes the proof of \autoref{thmes}.

\begin{rem}
\begin{enumerate}[label={(\arabic*)}]
\item The dimension of $P_{\Ga}^d\left(\RR^8, \RR^8 \right)$ is at most the dimension of $P_{\Gab}^d\left(\RR^8, \RR^8 \right)$ with $\Gab \le \Ga$. As a consequence we obtain $P_{\Ga}^3\left(\RR^8, \RR^8 \right) = P_{\Gab}^3\left(\RR^8, \RR^8 \right)$.
\item A similar statement holds true for the matrices $C$. We define the matrix $c(\phi)$ to be the representing matrix of $[e^{\phi \ii},1]$ and $\tilde{C}(\phi, \phi')$ as the diagonal blockmatrix of $c(\phi), c(\phi')$ for arbitrary distinct angles $\phi, \phi' \in S^1$. This leads to a compact $3$-dimensional Lie group
\[ \tilde{\mathbf{G}} = \left\langle Q, S, V, \tilde{C}(\phi, \phi'), D(\psi) \mid \phi, \phi', \psi \in S^1 \right\rangle. \]
It is easy to see that $E_1, \ldots, E_5$ are equivariant with respect to $\tilde{\mathbf{G}}$. Therefore the space $P_{\tilde{\mathbf{G}}}^3\left(\RR^8, \RR^8 \right)$ is also generated by these vectorfields. But as mentioned before $\tilde{\mathbf{G}}$ is not obtained from the closure of the union of all \Ga.
\item The Lie group $\tilde{\mathbf{G}}$ contains the matrices $C(a)$ and $D(b)$ for arbitrary values of $a$ and $b$. This could have served as proof for the \Gab-equivariance of $E_1 ,\ldots, E_5$.
\item At first it may appear odd that the number of linearly independent cubic equivariant polynomials is different in the case $a=5$. But from equivariance with respect to the Lie group $\tilde{\mathbf{G}}$ it follows that the dimension of cubic equivariants has to become stationary for some value of $a$. This occurs at the first step from $a=5$ to $a=13$. Therefore, when investigating \Gab -symmetric dynamical systems, we need to take care of the case $a=5$ separately.
\end{enumerate}
\end{rem}

%% file: lseb_bi_v2.tex

\section{Generic symmetry breaking bifurcations}
\label{ssma_mr}

In this section we want to investigate bifurcation problems on $\RR^8$ which are symmetric with respect to the groups \Gab that we have constructed before. In order to do so, we use methods proposed by \autocite{field1996symmetry, field2007dynamics} and \autocite{field1990symmetry}. The authors use techniques from equivariant transversality to develop a complete geometric theory on equivariant dynamics. It allows us, similar to the equivariant branching lemma, to obtain results on bifurcations in generic equations that are symmetric with respect to a given representation. The basic principle is that it suffices to investigate Taylor expansions up to some critical degree to gather information on the dynamical behaviour. Since these polynomials are equivariant as well, we can apply methods from invariant theory to calculate possible terms in the expansion, which is what we have done in the previous section using formula \eqref{eqsat2}. The authors even prove that we can always find such a critical degree $d$ in which the branching of solutions is fully determined. We say that the equivariant bifurcation problems are \emph{$d$-determined}. However we will not go that far here, as we see that the cubic truncation suffices to prove the bifurcation result. For this reason we do not try to establish determinacy statements. In our case we can apply a \emph{polar blowing-up} technique from the texts mentioned above to find a nontrivial branch of solutions bifurcating off the trivial one. All the methods used in this section are  formulated in Chapter 4 of \autocite{field2007dynamics}, where we can also find the technical details that we partly omit here. Furthermore we use a slight modification of this approach which respects the restriction on fixed point spaces of isotropy subgroups. This will be pointed out explicitly when we make use of it. As the equivariant structure forms a special case for $a=5$ (compare to the previous section), we restrict ourselves to $a \in \A$ with $a > 5$ for the rest of this section.

\subsection{Normalized families of equivariant vector fields}
Following the notation of \autocite{field2007dynamics} we let ${\mathcal{V}\left(\RR^8,G\right) = \CG}$ be the set of smooth $G$-equivariant vector fields, for $G\in \GG_a$ and $5<a\in \A$, depending on a real parameter. The action of $G$ on the product space is defined to be only on the first component. We equip the function space with the $C^\infty$-topology and subsets with the induced topology. For $f\in \mathcal{V}\left(\RR^8, G \right)$ we define the $1$-parameter family $\lbrace f_\lambda \rbrace_\lambda$ of smooth $G$-symmetric vector fields on $\RR^8$ by $f_\lambda = f( \cdot, \lambda)$. By equivariance we get 
\begin{align*}
f(0,\lambda) 		&= 0 \quad \text{for every } \lambda\in\RR,\\
D_1f(0,\lambda) 	& = \sigma_f(\lambda) \mathbbm{1} _8 
\end{align*}
with $\sigma_f \in C^\infty\left(\RR\right)$. This set of zeros will be called the branch of \emph{trivial zeros} and we are looking for solution branches bifurcating off this branch as we vary $\lambda$. As long as $\sigma_f(\lambda) \ne 0$ we can use the implicit function theorem to obtain a neighbourhood $U$ of $(0,\lambda)$ such that the only zeros in $U$ are trivial. We are therefore interested in points $\lambda_0 \in \RR$ with $\sigma_f(\lambda_0)=0$ to find nontrivial solutions. Generically in such a point $f$ will satisfy $\sigma_f'(\lambda_0) \ne 0$ which we will assume from now on. Furthermore we can assume $\lambda_0 = 0$ without loss of generality and use the inverse function theorem to reparametrize $\lambda$ so that $\sigma_f(\lambda) = \lambda$ for $\lambda$ near $0$. The extension of $\sigma_f$ to all the real numbers in the same manner does not impose a loss of generality, since we are only interested in branching close to the trivial solution. These considerations motivate the restriction to the closed affine linear subspace
\[ \VV = \VV \left(\RR^8, G\right) = \left\lbrace f \in \mathcal{V} \left( \RR^8, G \right) \mid \sigma_f
(\lambda) = \lambda, \lambda \in \RR \right\rbrace \]
of normalized families of smooth $G$-equivariant vector fields on $\RR^8$. For $f\in \VV$ we may write
\[ f_\lambda (x) = f(x,\lambda) = \lambda x + F_\lambda(x)\]
using Taylor's theorem, to which we refer as \emph{normalized bifurcation problem}.

\subsection{Nonradial equivariant polynomial maps}
To follow the methods of Field the next step is to find $G$-equivariant polynomial maps on $\RR^8$ which are nonradial. A polynomial map $P$ is called \emph{radial} if it is of the form
\[ P = p \mathbbm{1}_8 \]
where $p \colon \RR^8 \to \RR$ is an invariant polynomial. We call $d=d(G,\RR^8)$ the smallest degree in which nonradial equivariant polynomial maps exist. We have seen before from the Molien coefficients (Section \ref{ssma_es}) and \autoref{appce} in the Appendix that $d = 3$. For $f\in \VV$ let $R$ be the Taylor polynomial of $f_0$ of order $d$ at the origin. Using Taylor's theorem we obtain
\begin{equation*}
\label{normfam}
f(x,\lambda) = \lambda x + R(x) + F_1 (x) + \lambda F_2 (x,\lambda)
\end{equation*}
with $F_1(x) = f(x,0)-R(x) = \mathcal{O}(\|x\|^{d+1})$ and $F_2(x,\lambda) = f(x,\lambda) - R(x) - F_1(x) = \mathcal{O}(\|x\|^d)$ (see \autocite{field2007dynamics}).

It is a well known fact (that can be recalled from the results of Chapter 5 in \autocite{chossat2000methods}) that the homogeneous terms in the Taylor expansion of  an equivariant vector field are equivariant as well. The linear part $\lambda \mathbbm{1}_8$ of the Taylor polynomial of $f_\lambda$ vanishes for $\lambda = 0$. If we look again at the Molien coefficients, we find that there are no quadratic $G$-equivariant polynomial maps. Therefore $R \in P_{G}^3\left(\RR^8, \RR^8 \right)$ which is generated by $E_1, \ldots , E_5$ and hence $R$ must be of the form
\[ R= \alpha E_1 + \beta E_2 + \gamma E_3 + \delta E_4 + \epsilon E_5 \quad \text{with} \quad \alpha, \beta, \gamma, \delta, \epsilon \in \RR. \]

\subsection{Phase vector fields and fixed point subspaces}
As mentioned in the introduction to this section, the major technical tool to find nontrivial solution branches is a polar blowing-up technique. Using general polar coordinates we can decompose the normalized function $f\in \VV$ into a spherical part -- a smooth vector field on the unit sphere -- and a radial part perpendicular to the sphere. A suitable solution for the spherical part with the radial coordinate $0$ can then be generalized to other radial values, using the implicit function theorem, leading to nontrivial solutions. But first, we adapt Field's method in such a way, that it applies in fixed point spaces of isotropy subgroups. In our case this reduces the dimension by four, which is convenient for the following computations.

As we have seen before (Theorem \ref{thmgroups}) each group $G \in \GG_a$ has precisely one nontrivial isotropy type $[K]$ which contains the conjugate subgroups of
\[ K = \left\langle QS \right\rangle = \left\langle \diag{1,-1,1,-1,-1,1,-1,1} \right\rangle.\]
The subgroup $K$ obviously fixes elements of the subspace
\[ \fix{K} = \lbrace x \in \RR ^8 \mid x_2=x_4=x_5=x_7=0 \rbrace \cong \RR^4.\]
Utilizing the symmetry property of $f$, it therefore suffices to consider $\fix{K}$. Denote the coordinates by
\[ y = \left(y_1, y_2, y_3, y_4 \right) \in \fix{K}. \]
By equivariance $f_\lambda$ fixes $\fix{K}$ for each $f\in \VV$:
\[ kf_\lambda(x)  =f_\lambda(kx) = f_\lambda(x) \]
for $k \in K$ and  $x\in \fix{K}$. Therefore $f_\lambda(x) \in \fix{K}$ for all $x\in \fix{K}$. With this we may now restrict $f_\lambda$ and $R$ to $\fix{K}$, which is the part that does not appear in the texts by Field and Field \& Richardson. Since the blow-up method does not interfere with this reduction, we may perform it in the fixed point space just as well.

To investigate bifurcation behaviour in $\fix{K}$ we calculate the so called \emph{phase vector field} of the cubic equivariant polynomial maps. For $R$ as before restricted to $\fix{K}$ it is defined as the vector field
\[ \pvf{R}(y) = R(y) - \left\langle R(y), y \right\rangle y \quad \text{with} \quad y \in S^3 \subset \fix{K}\]
which is the tangential component of the restriction of $R$ to the unit sphere $S^3$ in $\RR^4$. Up to a factor depending on the radial coordinate, it is equal to the spherical part of $R$. Furthermore the phase vector field $\pvf{R}$ coincides with the spherical part of $f$ if the radial coordinate is $0$, which is the starting point for the blow-up technique. Note that the phase vector field of a radial polynomial vanishes and therefore cannot provide any information on solutions of the original equation. The projection on the phase vector field is a linear map from $P_{\Gab}^3\left(\RR^4, \RR^4 \right)$ to $P_{\Gab}^5\left(S^3, S^3 \right)$ so 
\[\pvf{R} = \alpha \pvf{E_1} + \beta \pvf{E_2} + \gamma \pvf{E_3} + \delta \pvf{E_4} + \epsilon \pvf{E_5}. \]
The phase vector fields of $E_1, \ldots E_5$ can be found in the appendix (see \autoref{apppfv}).
As we have seen in \autoref{thmes} the cubic equivariants $E_1, \ldots, E_5$ are \Ga-symmetric as well and the same holds for $\pvf{E_1}, \ldots, \pvf{E_5}$. Hence they leave the fixed point spaces of isotropy subgroups of \Ga invariant. In \autoref{liegroup} we have proved that \Ga has subgroups with one-dimensional fixed point spaces. These intersect the sphere in two points and therefore directly lead to zeros of the phase vector fields $\pvf{E_1}, \ldots, \pvf{E_5}$. For example the group
\[ \left\langle CD\left( - \frac{\pi}{a} \right)^{\rho}, QS \right\rangle < \Ga \]
fixes the one-dimensional subspace $\langle (0, \ldots, 0, 1)^T \rangle \subset \RR^8$. This is a subspace of $\fix{K}$ as well and reads $\langle (0,0,0,1)^T \rangle$ in the corresponding coordinates. Thus $y_0 = (0,0,0,1)^T$ and $-y_0$ are common zeros of $\pvf{E_1}, \ldots, \pvf{E_5}$ and therefore $\pvf{R}(\pm y_0) = 0$ for any linear combination. The Jacobian of $\pvf{R}(y_0)$ has the eigenvalues $-\alpha + \delta, -\alpha+\gamma, -\alpha+\beta$ and $-2\alpha$.
So we see that $y_0$ is a hyperbolic zero of $\pvf{R}$  if $\alpha \ne 0$, $\alpha \ne \beta$, $\alpha \ne \gamma$ and $\alpha \ne \delta$. These conditions are met for an open and dense subset of $\RR^5$ and therefore $y_0$ is generically a hyperbolic zero for $\pvf{R}$. This allows us to start the blow-up technique which, using the implicit function theorem, provides us with nontrivial hyperbolic solutions to $f = 0$ depending on the value of the radial coordinate. These can be reformulated into a solution curve bifurcating off the trivial solution where the direction of branching is $y_0$. By construction this new branch of solutions lies in the fixed point space $\fix{K}$ meaning that the isotropy type $[K]$ is symmetry breaking. For the technical details of the blow-up method see Lemma 4.8.1. and its proof from \autocite{field2007dynamics} with the fact that $\langle R(y_0), y_0 \rangle = \alpha$ which is generically not zero. As we have seen, the branching of steady states occurs for a generic bifurcation problem. This completes the proof for the main theorem on bifurcations with \Gab -symmetry.

%% file: lseb_sc.tex
\section{The special case $a=5$}
To conclude the above considerations we want to briefly discuss the special case $a=5$ and point out that it is not so special after all. As we have seen in \autoref{ssma_es} the main difference between the groups $G_{5,b}$ and \Gab (for admissible values of $b$) lies in the structure of equivariant polynomial maps. This in turn influences the argumentation to prove the bifurcation result. In \autoref{subsec_es} we have computed the space of equivariant cubic polynomial maps $P_{G_{5,b}}^3 \left(\RR^8, \RR^8 \right)$ to be generated by the maps $E_1, \ldots, E_8$ (see \autoref{appce}). Furthermore, in \autoref{subsec_gc}, we have seen that the corresponding spaces are generated by $E_1,\ldots,E_5$ whenever $a>5$. A major aspect for the proof of the bifurcation result is the fact that these maps are equivariant not only with respect to the groups \Gab but also with respect to the Lie groups \Ga. A short calculation shows that this holds true for the maps $E_6, E_7$ and $E_8$ as well. But, on the contrary to the case $a>5$, the additional maps are not equivariant with respect to the largest Lie group $\tilde{\mathbf{G}}$ that we have considered.

Nevertheless we may use the same technique to investigate the bifurcation behavior in the presence of $G_{5,b}$ symmetry as before. We only sketch the proof here since it is completely analog to the one before. We consider the cubic truncation of a normalized bifurcation problem
\[
R =  \sum_{i=1}^{8} \alpha_i E_i \quad \text{with} \quad a_i \in \RR
\]
and restrict to the fixed point subspace
\[ \fix{K} = \lbrace x \in \RR ^8 \mid x_2=x_4=x_5=x_7=0 \rbrace \cong \RR^4.\]
Then we consider the corresponding phase vector field
\[ \pvf{R}(y) = R(y) - \left\langle R(y), y \right\rangle y \quad \text{with} \quad y \in S^3 \subset \fix{K}\]
which by the same argumentation as before -- one-dimensional fixed point space of an isotropy subgroup of $\mathbf{G}_5$ -- has the zero $y_0 = (0,0,0,1)^T$. This is once again generically hyperbolic and thus we can apply the polar blowing up method to obtain a branch of zeros for the bifurcation equation bifurcating off the trivial solution.

Summing up we see that the case $a=5$, even though it has to be treated separately, does not imply significant differences. It provides the same bifurcation result which can be proved using the same techniques. The main point of interest lies in the structure of the equivariant maps as $P_{\Gab}^3\left(\RR^8, \RR^8 \right) \subset P_{G_{5,b}}^3\left(\RR^8, \RR^8 \right)$ as a proper subspace.

%% file: lseb_fg_v2.tex

\section{Further groups and even dimensional representations}
\label{ssma_fg}

In this paper we have constructed groups of order $16\m$ -- here $\m=a\cdot b$ with $a,b \ne 1$, relatively prime, odd and $a>5$ being a product of prime powers of the form $1 \mod 4$ -- with an $8$-dimensional absolutely irreducible representation that provide counterexamples to the Ize conjecture. Furthermore \autocite{lauterbach2014equivariant} describes groups of order $64+128\ell$ with $\ell\in \NN$ with the same purpose. In both cases there are generically symmetry breaking isotropy types. But comparing these results to the GAP-calculations provided by \mbox{Table 7} in \autocite{lauterbach2010do} we still expect a vast number of counterexamples to the Ize conjecture in $8$ dimensions that have not yet been investigated systematically. It is an open task to find a reasonable ordering for the groups in terms of their orders and to obtain information on the dynamics in their $8$-dimensional representations.

In order to do so, a first step could be to slightly adapt the construction of \Gab in such a way that we define the $8\times8$ generating matrix $Q$ to be
\[ Q = \begin{pmatrix}
q	& 0 \\
0	& q
\end{pmatrix}\] 
instead of the second block being $-q$. We obtain two isotropy types in this case for which it would be interesting to determine whether both of them are generically symmetry breaking.  Another task is to determine the role of the $3$-dimensional Lie group containing all the \Ga. This has not yet been sufficiently investigated.

Furthermore we see from Tables 5-10 in \autocite{lauterbach2010do} that there are further groups acting absolutely irreducibly in dimensions $4,8,12,16$ and $20$ that appear to lead to counterexamples to the Ize conjecture but none with the same property in dimensions $2,6,10,14$ and $18$ (at least for small group orders). The authors formulate the conjecture: \emph{``For dimensions $N = 0 \mod 4$, there are infinitely many groups acting absolutely irreducibly on $\RR^N$ that have no isotropy subgroups with odd-dimensional fixed point spaces. But for dimensions $N = 2 \mod 4$, there are no such groups''}. There is some evidence for this conjecture to be true as the GAP calculations do not provide counterexamples in the second case for groups of order up to $1000$. Furthermore \autocite{ruan2011fixed} proves the claim for dimension $6$ under the mild additional assumption that the groups are solvable.

These intermediate steps and the conjecture of \autocite{lauterbach2010do} would provide some major insight in the question if absolute irreducible group actions lead to generically symmetry breaking isotropy types. This interpretation of the Ize conjecture from the dynamical systems point of view would be a significant contribution to the understanding of bifurcations in the presence of symmetry. However such a general statement is still far from being proved.

\renewcommand{\abstractname}{Acknowledgements}
\begin{abstract}
\noindent
R.L. would like to thank U. Kühn for some helpful discussions on modular \mbox{congruences.}
\end{abstract}

%% file: lseb_ap_v2.tex

\begin{appendix}
\renewcommand\thetheorem{\thesubsection.\arabic{theorem}}
\makeatletter
\@addtoreset{theorem}{subsection}
\makeatother
\renewcommand\theequation{\thesubsection.\arabic{equation}}
\makeatletter
\@addtoreset{equation}{subsection}
\makeatother
\section{Appendix}

\subsection{Calculations of Molien coefficients}
\label{appsatt}

In this section we want to fill the gaps that were left in Section \ref{ssma_es} in the calculations of 
\[ R_d = \frac{1}{\left| \Gab \right|} \sum_{g \in \Gab} \chi_{(d)} (g) \chi (g) \]
for $d=2,3$. First of all let
\[ \eta = \frac{k_1}{a} \pi \quad \text{and} \quad \nu = \frac{k_2}{b} \pi \]
and note that
\begin{equation*} 
\chi_{(2)} (g) = 2 \left( \cos \left( 2 \eta \right) + \cos \left(2 \rho \eta \right) \right) \cos \left(2 \nu \right) 
+ 8 \left( \cos \left(\eta \right) + \cos \left(\rho \eta \right) \right)^2 \cos \left(\nu \right)^2 
\end{equation*}
which only depends on $k_1$ and $k_2$. Remember that the only nonzero terms occur for $l_1 \in \lbrace 0,2 \rbrace$ and $l_2 = m = 0$. This allows us to calculate
\[
R_2 = \frac{1}{16ab} \sum_{k_1 = 0}^{a-1} \sum_{k_2 = 0}^{b-1} \sum_{l_1 \in \lbrace 0,2 \rbrace } (-1)^{\frac{l_1}{2}} 4 \chi_{(2)}(g) \left( \cos \left(\eta \right) + \cos \left(\rho \eta \right) \right) \cos \left(\nu \right).
\]
Because of the factor $(-1)^{\frac{l_1}{2}}$  whereas the rest of each summand is independent of $l_1$ these two summands cancel each other. This directly yields

\begin{prop}
There are no quadratic equivariant maps for the $8$-dimensional representation of \Hab:
\[R_2 = 0. \]
\end{prop}

The case $d=3$ is much more complicated. We compute
\begin{align*}
\chi_{(3)} (g) 
	&= (-1)^{\frac{l_1}{2}} \frac{32}{3} \left( \cos \left(\eta \right) + \cos \left(\rho \eta \right) \right)^3 
	 \cos \left(\nu \right)^3 \\
	&\phantom{=} \mathrel{+} (-1)^{\frac{l_1}{2}} 8 \left( \cos \left(2 \eta \right) + \cos \left(2 \rho \eta \right) \right) 
	 \left( \cos \left(\eta \right) + \cos \left(\rho \eta \right) \right) \cdot \cos \left(2\nu \right) \cos \left(\nu \right) \\
	&\phantom{=} \mathrel{+}  (-1)^{\frac{l_1}{2}} \frac{4}{3} \left( \cos \left(3 \eta \right) + \cos \left(3 \rho \eta \right) \right) 
	 \cos \left(3\nu \right).
\end{align*}
This depends on $l_1$ therefore the terms do not cancel out as easily as in the case $d=2$. To be able to compute $R_3$, we state two technical lemmas first.

\begin{lemma}
Let $w \in \NN$ and $l \in \ZZ$. Then
\[ \sum_{k=0}^{w-1} \cos 2l\frac{k}{w} \pi = 
\begin{cases}
0	&\quad \text{for} \quad w \nmid l , \\
w	&\quad \text{else}.
\end{cases} \]
\end{lemma}

\noindent
We want to use this lemma to calculate sums of such cosine terms that contain an even factor in front of $(k/w) \pi$. We have to distinguish whether this factor is an integer multiple of $2w$. The following lemma performs this distinction in the occurring cases.

\begin{lemma}
\label{cong}
Let $a \in \A$ and $\rho$ be chosen as in Proposition \ref{prime} and odd (compare to the construction of \Gab). Then
\begin{enumerate}[label={(\arabic*)},ref={\thecor~(\arabic*)}]
\item \label{cong1} $\rho-1 \ne 0 \mod 2a$;
\item \label{cong2} $2(\rho-1) \ne 0 \mod 2a$;
\item \label{cong3} $\rho+1 \ne 0 \mod 2a$;
\item \label{cong4} $2(\rho+1) \ne 0 \mod 2a$;
\item \label{cong5} $2\rho \ne 0 \mod 2a$;
\item \label{cong6} $4\rho \ne 0 \mod 2a$;
\item \label{cong7} $\rho-3 = 0 \mod 2a$ if and only if $a = 5$ and $\rho=3$;
\item \label{cong8} $3\rho-1 \ne 0 \mod 2a$;
\item \label{cong9} $\rho+3 \ne 0 \mod 2a$;
\item \label{cong10} $3\rho+1 = 0 \mod 2a$ if and only if $a = 5$ and $\rho=3$.
\end{enumerate}
\end{lemma}
\begin{proof}
\begin{enumerate}[label={(\arabic*)}]
\item Suppose $\rho -1 = 0 \mod 2a$. Then $\rho = 1 \mod a$ and $\rho^2=1 \mod a$ which is a contradiction to the choice of $\rho$.
\item Suppose $2(\rho -1) = 0 \mod 2a$. Then $\rho -1 = 0 \mod a$ and $\rho = 1 \mod a$. The contradiction follows as before.
\item Suppose $\rho + 1 = 0 \mod 2a$. Then $\rho = -1 \mod a$ and $\rho^2 = 1 \mod a$ which is again a contradiction.
\item Suppose $2(\rho +1) = 0 \mod 2a$. Then $\rho +1 = 0 \mod a$ and $\rho = -1 \mod a$. This is a contradiction as before.
\item Suppose $2 \rho = 0 \mod 2a$. Then $\rho = 0 \mod a$ which contradicts the choice of $\rho$.
\item Suppose $4 \rho = 0 \mod 2a$. Then $2\rho = 0 \mod a$ and therefore $4 \rho^2 = 0 \mod a$. But $\rho$ was chosen so that $4 \rho^2 = -4 \mod a$. This is impossible, since $a$ is odd.
\item Suppose $\rho - 3 = 0 \mod 2a$. Then $\rho = 3 \mod a$ and $\rho^2 = 9 \mod a$. But $\rho^2 = -1 \mod a$ so $9 = -1 \mod a$. The only choice is $a = 5$, since $a$ is odd. The corresponding odd $\rho$ is $3$.
\item Suppose $3 \rho -1 = 0 \mod 2a$. Then $3\rho = 1 \mod a$ and $\rho = 3\rho^2 = -3 \mod a$. This yields $\rho^2 = 9 \mod a$ and as before $a=5$. Then $\rho = -3 = 2 \mod a$ which is not odd and therefore a contradiction.
\item Suppose $\rho + 3 = 0 \mod 2a$. Then $\rho = -3 \mod a$ and the contradiction follows as before.
\item Suppose $3 \rho + 1 = 0 \mod 2a$. Then $3 \rho = -1 \mod a$ and therefore $\rho = -3 \rho^2 = 3 \mod a$. As before we obtain $a = 5$ and $\rho =3$.
\end{enumerate}
\end{proof}

\noindent
These two technical lemmas together with multiple applications of the calculation rules for cosine allow us to compute $R_3$ explicitly and we obtain

\begin{prop}
There are eight linearly independent cubic equivariant maps for the $8$-dimensional representation of \Gab if $a=5$ and five linearly independent cubic equivariant polynomial maps for all other $a \in \A$:
\[R_3 = \begin{cases}
	8	&\quad \text{for} \quad a=5, \\
	5	&\quad \text{else}. \\
\end{cases} \]
\end{prop}

\clearpage

\subsection{Cubic equivariant maps}

\begin{table}[!ht]
\begin{center}
\begin{tabular}{ll}
	$ E_1(x) = \begin{pmatrix}
		\left(x_1^2 + x_2^2 \right) x_1 \\ 
		\left(x_1^2 + x_2^2 \right) x_2 \\
		\left(x_3^2 + x_4^2 \right) x_3 \\
		\left(x_3^2 + x_4^2 \right) x_4 \\
		\left(x_5^2 + x_6^2 \right) x_5 \\
		\left(x_5^2 + x_6^2 \right) x_6 \\
		\left(x_7^2 + x_8^2 \right) x_7 \\
		\left(x_7^2 + x_8^2 \right) x_8
	\end{pmatrix} $ & 
	$ E_5(x) = \begin{pmatrix}
	-x_3x_5x_7 - x_3x_6x_8 - x_4x_5x_8 + x_4x_6x_7 \\ 
	x_3x_5x_8 - x_3x_6x_7 - x_4x_5x_7 - x_4x_6x_8 \\
	-x_1x_5x_7 - x_1x_6x_8 + x_2x_5x_8 - x_2x_6x_7 \\
	-x_1x_5x_8 + x_1x_6x_7 - x_2x_5x_7 - x_2x_6x_8 \\
	x_1x_3x_7 + x_1x_4x_8 - x_2x_3x_8 + x_2x_4x_7 \\
	x_1x_3x_8 - x_1x_4x_7 + x_2x_3x_7 + x_2x_4x_8 \\
	x_1x_3x_5 - x_1x_4x_6 + x_2x_3x_6 + x_2x_4x_5 \\
	x_1x_3x_6 + x_1x_4x_5 - x_2x_3x_5 + x_2x_4x_6
	\end{pmatrix} $ \\
	$ E_2(x) = \begin{pmatrix}
	\left(x_3^2 + x_4^2 \right) x_1 \\ 
	\left(x_3^2 + x_4^2 \right) x_2 \\
	\left(x_1^2 + x_2^2 \right) x_3 \\
	\left(x_1^2 + x_2^2 \right) x_4 \\
	\left(x_7^2 + x_8^2 \right) x_5 \\
	\left(x_7^2 + x_8^2 \right) x_6 \\
	\left(x_5^2 + x_6^2 \right) x_7 \\
	\left(x_5^2 + x_6^2 \right) x_8
	\end{pmatrix} $ &
	$ E_6(x) = \begin{pmatrix}
	-x_1x_3x_6 + x_1x_4x_5 + x_2x_3x_5 + x_2x_4x_6 \\ 
	x_1x_3x_5 + x_1x_4x_6 + x_2x_3x_6 - x_2x_4x_5 \\
	-x_1x_3x_8 + x_1x_4x_7 + x_2x_3x_7 + x_2x_4x_8 \\
	x_1x_3x_7 + x_1x_4x_8 + x_2x_3x_8 - x_2x_4x_7 \\
	x_3x_5x_8 + x_3x_6x_7 + x_4x_5x_7 - x_4x_6x_8 \\
	x_3x_5x_7 - x_3x_6x_8 - x_4x_5x_8 - x_4x_6x_7 \\
	-x_1x_5x_8 - x_1x_6x_7 - x_2x_5x_7 + x_2x_6x_8 \\
	-x_1x_5x_7 + x_1x_6x_8 + x_2x_5x_8 + x_2x_6x_7
	\end{pmatrix} $ \\
	$ E_3(x) = \begin{pmatrix}
	\left(x_5^2 + x_6^2 \right) x_1 \\ 
	\left(x_5^2 + x_6^2 \right) x_2 \\
	\left(x_7^2 + x_8^2 \right) x_3 \\
	\left(x_7^2 + x_8^2 \right) x_4 \\
	\left(x_3^2 + x_4^2 \right) x_5 \\
	\left(x_3^2 + x_4^2 \right) x_6 \\
	\left(x_1^2 + x_2^2 \right) x_7 \\
	\left(x_1^2 + x_2^2 \right) x_8
	\end{pmatrix} $ &
	$ E_7(x) = \begin{pmatrix}
	-2x_5x_7x_8 - x_6x_7^2 + x_6x_8^2 \\ 
	-x_5x_7^2 + x_5x_8^2 + 2x_6x_7x_8 \\
	x_5^2x_8 + 2x_5x_6x_7 - x_6^2x_8 \\
	x_5^2x_7 - 2x_5x_6x_8 - x_6^2x_7 \\
	x_1^2x_4 + 2x_1x_2x_3 - x_2^2x_4 \\
	-x_1^2x_3 + 2x_1x_2x_4 + x_2^2x_3 \\
	2x_1x_3x_4 + x_2x_3^2 - x_2x_4^2 \\
	-x_1x_3^2 + x_1x_4^2 + 2x_2x_3x_4
	\end{pmatrix} $ \\
	$ E_4(x) = \begin{pmatrix}
	\left(x_7^2 + x_8^2 \right) x_1 \\ 
	\left(x_7^2 + x_8^2 \right) x_2 \\
	\left(x_5^2 + x_6^2 \right) x_3 \\
	\left(x_5^2 + x_6^2 \right) x_4 \\
	\left(x_1^2 + x_2^2 \right) x_5 \\
	\left(x_1^2 + x_2^2 \right) x_6 \\
	\left(x_3^2 + x_4^2 \right) x_7 \\
	\left(x_3^2 + x_4^2 \right) x_8
	\end{pmatrix} $ &
	$ E_8(x) = \begin{pmatrix}
	x_3^2x_8 - 2x_3x_4x_7 - x_4^2x_8 \\ 
	-x_3^2x_7 -2x_3x_4x_8 + x_4^2x_7 \\
	x_1^2x_6 - 2x_1x_2x_5 - x_2^2x_6 \\
	-x_1^2x_5 - 2x_1x_2x_6 + x_2^2x_5 \\
	2x_1x_7x_8 + x_2x_7^2 - x_2x_8^2 \\
	x_1x_7^2 - x_1x_8^2 - 2x_2x_7x_8 \\
	-2x_3x_5x_6 - x_4x_5^2 + x_4x_6^2 \\
	-x_3x_5^2 + x_3x_6^2 + 2x_4x_5x_6
	\end{pmatrix} $
\end{tabular}
\end{center}
\captionsetup{width=.95\textwidth}
\caption{Cubic quivariant maps $E_1, \ldots, E_8$ for $G_{5,3}$.}
\label{appce}
\end{table}
\clearpage

\subsection{Phase vector fields}

\begin{table}[!ht]
\begin{center}
\begin{tabular}{l}
	$ \pvf{E_1} (y)	= \begin{pmatrix} -y_{1}\, \left( {y_{1}}^{4}+{y_{2}}^{4}+{y_{
				3}}^{4}+{y_{4}}^{4}-{y_{1}}^{2} \right) \\ \noalign{\medskip}-y_{2}\, \left( {y_{1}}^{4}+{y_{2}}^{4}+{y_{3}}^{
			4}+{y_{4}}^{4}-{y_{2}}^{2} \right) \\ \noalign{\medskip}-y_{3}\, \left( {y_{1}}^{4}
		+{y_{2}}^{4}+{y_{3}}^{4}+{y_{4}}^{4}-{y_{3}}^{2} \right) 
		\\ \noalign{\medskip}-y_{4}\, \left( {y_{1}}^{4}
		+{y_{2}}^{4}+{y_{3}}^{4}+{y_{4}}^{4}-{y_{4}}^{2} \right) \end{pmatrix} $ \\
		$ \pvf{E_2} (y)	= \begin{pmatrix} -y_{1}\, \left( 2\,{y_{1}}^{2}{y_{2}}^{2}+2
		\,{y_{3}}^{2}{y_{4}}^{2}-{y_{2}}^{2} \right)
		\\ \noalign{\medskip}-y_{2}\, \left( 2\,{y_{1}}^{2}{y_{2}}^{2}+2\,{y_{
				3}}^{2}{y_{4}}^{2}-{y_{1}}^{2} \right) \\ \noalign{\medskip}-y_{3}\, \left( 2\,{y_{1}}^
		{2}{y_{2}}^{2}+2\,{y_{3}}^{2}{y_{4}}^{2}-{y_{4}}^{2} \right) 
		\\ \noalign{\medskip}-y_{4}\, \left( 2\,{y_{1}}^
		{2}{y_{2}}^{2}+2\,{y_{3}}^{2}{y_{4}}^{2}-{y_{3}}^{2} \right) 
		\end{pmatrix} $ \\
		$ \pvf{E_3} (y)	= \begin{pmatrix} -y_{1}\, \left( {y_{1}}^{2}{y_{3}}^{2}+{y_{4
			}}^{2}{y_{1}}^{2}+{y_{3}}^{2}{y_{2}}^{2}+{y_{2}}^{2}{y_{4}}^{2}-{y_{3}
		}^{2} \right) \\ \noalign{\medskip}-y_{2}\,
		\left( {y_{1}}^{2}{y_{3}}^{2}+{y_{4}}^{2}{y_{1}}^{2}+{y_{3}}^{2}{y_{2
			}}^{2}+{y_{2}}^{2}{y_{4}}^{2}-{y_{4}}^{2} \right) 
			\\ \noalign{\medskip}-y_{3
			}\, \left( {y_{1}}^{2}{y_{3}}^{2}+{y_{4}}^{2}{y_{1}}^{2}+{y_{3}}^{2}{y
			_{2}}^{2}+{y_{2}}^{2}{y_{4}}^{2}-{y_{2}}^{2} \right) 
		\\ \noalign{\medskip}-y_{4}\, \left( {y_{1}}^{2}
		{y_{3}}^{2}+{y_{4}}^{2}{y_{1}}^{2}+{y_{3}}^{2}{y_{2}}^{2}+{y_{2}}^{2}{
			y_{4}}^{2}-{y_{1}}^{2} \right) \end{pmatrix} $ \\
		$ \pvf{E_4} (y)	= \begin{pmatrix} -y_{1}\, \left( {y_{1}}^{2}{y_{3}}^{2}+{y_{1
			}}^{2}{y_{4}}^{2}+{y_{3}}^{2}{y_{2}}^{2}+{y_{2}}^{2}{y_{4}}^{2}-{y_{4}
		}^{2} \right) \\ \noalign{\medskip}-y_{2}\,
		\left( {y_{1}}^{2}{y_{3}}^{2}+{y_{1}}^{2}{y_{4}}^{2}+{y_{3}}^{2}{y_{2
			}}^{2}+{y_{2}}^{2}{y_{4}}^{2}-{y_{3}}^{2} \right) 
			\\ \noalign{\medskip}-y_{3
			}\, \left( {y_{1}}^{2}{y_{3}}^{2}+{y_{1}}^{2}{y_{4}}^{2}+{y_{3}}^{2}{y
			_{2}}^{2}+{y_{2}}^{2}{y_{4}}^{2}-{y_{1}}^{2} \right) 
		\\ \noalign{\medskip}-y_{4}\, \left( {y_{1}}^{2}
		{y_{3}}^{2}+{y_{1}}^{2}{y_{4}}^{2}+{y_{3}}^{2}{y_{2}}^{2}+{y_{2}}^{2}{
			y_{4}}^{2}-{y_{2}}^{2} \right) \end{pmatrix} $ \\
		$ \pvf{E_5} (y)	= \begin{pmatrix} -y_{2}\,y_{3}\,y_{4}
		\\ \noalign{\medskip}-y_{1}\,y_{3}\,y_{4}\\ \noalign{\medskip}y_{1}\,y_{2}\,y_{4}
		\\ \noalign{\medskip}y_{1}\,y_{2}\,y_{3}
		\end{pmatrix} $
\end{tabular}
\end{center}
\captionsetup{width=.95\textwidth}
\caption{Phase vector fields of $E_1, \ldots, E_5$ restricted to $S^3 \subset \fix{K}$.}
\label{apppfv}
\end{table}

\end{appendix}